\documentclass{amsart}
\usepackage{amssymb}
\usepackage{amscd}
\usepackage{amsmath}
\usepackage{epsfig}
\usepackage{graphicx}
\usepackage{color}
\usepackage{lmodern}

\usepackage{time}

\usepackage[all]{xy}

\usepackage{pgf,tikz}
\usepackage{caption}
\usepackage{shuffle}

\title[Analytic continuation]
{Analytic continuation of multiple polylogarithms in positive characteristic} 
\author{Hidekazu Furusho}
\address{Graduate School of Mathematics, Nagoya University, Chikusa-ku, Furo-cho, Nagoya, 464-8602,  Japan}
\email{furusho@math.nagoya-u.ac.jp}



\date{February 4, 2022}

\newtheorem{thm}{Theorem}[subsection]
\newtheorem{lem}[thm]{Lemma}
\newtheorem{cor}[thm]{Corollary}
\newtheorem{prop}[thm]{Proposition}  

\theoremstyle{remark}

\theoremstyle{definition}
\newtheorem{defn}[thm]{Definition}
\newtheorem{rem}[thm]{Remark}

\numberwithin{equation}{section}
\numberwithin{figure}{section}

\newcommand{\GL}{\mathrm{GL}}

\newcommand{\Lie}{\mathrm{Lie}}

\newcommand{\Q}{\mathbb{Q}}
\newcommand{\C}{\mathbb{C}}

\newcommand{\Z}{\mathbb{Z}}
\newcommand{\N}{\mathbb{N}}
\newcommand{\F}{\mathbb{F}}

\newcommand{\TT}{\mathbb{T}}
\newcommand{\M}{\mathbb{M}}

\newcommand{\T}{\mathtt{T}}
\newcommand{\E}{\mathcal{E}}

\newcommand{\Li}{\mathrm{Li}}
\newcommand{\LLii}{\mathfrak{Li}}

\newcommand{\Exp}{\mathrm{Exp}}
\newcommand{\Log}{\mathrm{Log}}

\newcommand{\Mat}{\mathrm{Mat}}

\begin{document}
\bibliographystyle{amsalpha+}
\maketitle

\begin{abstract}
Our aim of this paper is to propose
a method of analytic continuation of 
Carlitz multiple  (star) polylogarithms
to the whole space by using Artin-Schreier equation
and present a treatment of their branches
by introducing the notion of monodromy modules.
As applications of this method, we obtain
(1) 
a method of continuation 
of the logarithms of  
higher tensor powers of  Carlitz module,
(2) the orthogonal property %
(Chang-Mishiba functional relations), 
(3) 
a branch independency of the Eulerian property.
\end{abstract}

\tableofcontents
\setcounter{section}{-1}
\section{Introduction}\label{introduction}
It is said that the history of study of the  polylogarithm  
goes  back to the  the correspondence of Leibniz with Bernoulli in 
1696.
%
The {\it polylogarithm} 
is the complex function defined by the following series:
$$
\Li_n(z)=\sum_{i=1}^\infty\frac{z^i}{i^n}
$$
with a positive integer $n\geqslant 1$.
The case for $n=1$ gives $\Li_1(z)=-\log(1-z)$ and 
that for $n=2$ gives the dilogarithm.
Though it converges on $|z|<1$, it can be analytically continued to 
a bigger region, in precise a covering of 
${\mathbb P}^1(\C)\setminus\{0,1,\infty\}$,
by iterated path integrals.
It is significant in number theory that
its special value at $z=1$, that is, its limit value $z\to 1$, attains the Riemann zeta value 
$\zeta(n)=\sum_{i=1}^\infty\frac{1}{i^n}$ ($n>1$).
The function is generalized to 
the {\it multiple polylogarithm}
which is defined by the following series:  
$$
\Li_{n_1,\dots,n_d}(z_1,\dots,z_d)=\sum_{0< i_1<\cdots<i_d}
\frac{z_1^{i_1}\cdots z_d^{i_d}}{{i_1}^{n_1}\cdots {i_d}^{n_d}}
$$
with 
$n_1,\dots, n_d\geqslant 1$.
Though it converges when 
$|z_k|<1$
for $k=1,\dots,d$,
it can be analytically continued to a bigger region
by iterated integrals (cf. \cite{Z}).
It is remarkable that 
its special value at $z_1=\cdots=z_d=1$ gives the multiple zeta value
$$\zeta(n_1,\dots, n_d)=\sum_{0< i_1<\cdots<i_d}
\frac{1}{{i_1}^{n_1}\cdots {i_d}^{n_d}}$$ 
when $n_d>1$ (the condition to converge).

While in the case of the global function field in positive characteristic,
Carlitz introduced Carlitz zeta value $\zeta_C(n)$ ($n\geqslant 1$) around 1935, 
which is regarded to be an analogue of the Riemann zeta  value $\zeta(n)$.
Anderson and Thakur \cite{AT90} considered 
the {\it Carlitz polylogarithm}
(denoted $\Li_n(z)$ by abuse of notation)
as an analogue of the above polylogarithm, 
which is defined by  the series
\begin{equation*}\label{eq:CPL}
\Li_{n}(z)=\sum_{i=0}^\infty
\frac{z^{q^i}}{L_i^n}
\in \C_\infty[[z]]
\end{equation*}
(consult \S \ref{sec:preparation} for these symbols).
The function converges on  $|z|_\infty< q^\frac{nq}{q-1}$.
They  showed that $\zeta_A(n)$ is given by a certain linear combination of
its special value at some algebraic numbers lying on the region of convergence.
Thakur \cite{T-book} introduced an analogue $\zeta_A(n_1,\dots,n_d)$
($n_1,\dots,n_d\geqslant 1$) of
multiple zeta value which generalizes the Carlitz zeta value.
Chang \cite{C14} generalized the Carlitz polylogarithm to 
the {\it Carlitz multiple polylogarithm}
(denoted $\Li_{n_1,\dots,n_d}(z_1,\dots,z_d)$ by abuse of notation)
which is defined by the series 
\begin{equation*}\label{eq:CMPL}
\Li_{n_1,\dots,n_d}(z_1,\dots,z_d)=\sum_{0\leqslant i_1<\cdots<i_d}
\frac{z_1^{q^{i_1}}\cdots z_d^{q^{i_d}}}{L_{i_1}^{n_1}\cdots L_{i_d}^{n_d}}
\in \C_\infty[[z_1,\dots,z_d]]
\end{equation*}
in the region of convergence $\mathbb D$ 
(cf. \eqref{eq:D})
and he further showed that   $\zeta_C(n_1,\dots,n_d)$ is given
by a certain linear combination of
its special value at some algebraic numbers lying on $\mathbb D$.
Its star variant \eqref{eq:CMSPL} was introduced and discussed
in Chang-Mishiba \cite{CM}.
Its relationship with Anderson dual $t$-motives  and $t$-modules
is developed in \cite{CPY, CGM,GN}.

The aim of this paper is to extend the regions of convergence of
($t$-motivic) Carlitz multiple (star) polylogarithms by using Artin-Schreier
equations which serve as a substitute of iterated path integrals.
In \S \ref{sec:anacon}, we extend the functions 
by using Artin-Schreier equations and 
explain a manipulation of their associated branches
by introducing the notion of monodromy modules. 
In \S \ref{sec:applications},
by exploiting this method,
we  give a method of continuation 
of the logarithms of $t$-modules associated with  
higher tensor powers of Carlitz module, 
analytic continuation of
Chang-Mishiba functional relations, and 
a branch independency of the Eulerian property.

\section{Analytic continuation of Carlitz multiple polylogarithms}
\label{sec:anacon}
We explain a method of analytic continuation
of the Carlitz multiple (star) polylogarithm 
by using Artin-Schreier equation.
In \S \ref{sec:preparation}, we prepare the notations to be used and also present a key lemma (Lemma \ref{lem:AS}) related to Artin-Schreier equation.
In \S \ref{sec:CM(S)PL}, we recall the definition of 
the ($t$-motivic) Carlitz multiple (star) polylogarithm. 
In \S \ref{sec:prolong PL}, we explain a method of 
continuation of the Carlitz  polylogarithm.
By extending the method, we give an analytic continuation of
the Carlitz multiple polylogarithm in \S \ref{sec:prolong MPL}
and the Carlitz multiple star polylogarithm in \S \ref{sec:prolong MSPL}
both as one variable functions.

\subsection{Preparation}\label{sec:preparation}
In this paper the following notation is employed.
\begin{itemize}
\item $\N$: the set of positive integers
\item $\F_q$: the field with $q$ elements, for $q$ a power of a prime number $p$
\item $A=\F_q[\theta]$: the polynomial ring in the variable $\theta$ over $\F_q$
\item $A_+$: the set of monic polynomials in $A$, which is an analogue of the set of positive integers $\N=\Z_{>0}$
\item  $K$:  the fraction field of $A$ 
\item $\infty$: the infinite place of $K$ with an associated absolute value $|\cdot|_\infty$ such that $|\theta|_\infty=q$
\item $K_\infty=\F_q((1/\theta))$: the $\infty$-adic completion of $K$
\item $\C_\infty$: the $\infty$-adic completion of the algebraic closure $\bar K_\infty$
\item $\TT$: the Tate algebra with respect to another parameter $t$, the ring of 
formal power series $f=\sum a_it^i\in\C_\infty[[t]]$ convergent on $|t|_\infty\leqslant 1$,
encoded with the Gauss norm given by 
$||f||_\infty:=\max_i\{|a_i|_\infty\}$
\item $\TT_r$ ($r\in q^\Q$): the subalgebra of $\C_\infty[[t]]$ which converges on $|t|_\infty\leqslant r$, so $\TT=\TT_1$. 
\item 
$\TT(\infty)$: the intersection of $\TT_r$ for  all ${r\in q^\Q}$,
the set of formal power series
$\sum_{i=0}^\infty a_it^i\in\C_\infty[[t]]$ such that
$\lim_{n\to\infty}\sqrt[n]{|a_n|_\infty}=0$
\item $\E$: the ring of {\it entire functions}, that is,
formal power series
$f=\sum_{i=0}^\infty a_it^i\in\bar K[[t]]$ such that 
$f\in\TT(\infty)$
and $[K_\infty(a_0,a_1,a_2,\dots):K_\infty]<\infty$.
\item The {\it $n$-fold Frobenius twisting} ($n\in \Z$) on the field $\C_\infty((t))$
is defined by  $f=\sum_ia_it^i\in \C_\infty((t)) \mapsto
f^{(n)}=\sum_ia_i^{q^n}t^i\in \C_\infty((t))$
\item  $\wp:\TT\to\TT$ is the $\F_q[t]$-linear map sending $f\mapsto f-f^{(1)}$
\end{itemize}

The following lemma plays an essential role  in this section.

\begin{lem}\label{lem:AS}
\rm{(1).} The map $\wp:\TT\to\TT$ is surjective and
the inverse $\wp^{-1}(h)$ for each $h\in\TT$ is given by $h'+\F_q[t]$ for some  $h'\in\TT$.

\rm{(2).} For any $f\in\TT$, $f$ and $\wp(f)$ have a same radius of convergence.

\rm{(3).} If $V$ is an $\F_q[t]$-submodule of $\TT$, then so is $\wp^{-1}(V)$.

\rm{(4).} $\wp^{-1}(\E)=\E$.
\end{lem}

\begin{proof}
For $f=\sum_i a_it^i\in\TT$ with $a_i\in \C_\infty$,
we calculate its inverse image $g=\sum_ib_it^i$ by solving the 
following Artin-Schreier type equation
\begin{equation}\label{eq:AS}
b_i-b_i^q=a_i
\end{equation}
for each $i$. 
Though solutions of the above equation are  unique modulo $\F_q$ for each $i$,
we see that  $g$ is uniquely determined modulo $\F_q[t]$
because we impose the condition $g\in\TT$.
It is immediate to see that $g$ belongs to $\TT$
because we have
\begin{equation}\label{eq:norm equality}
|b_i|_\infty=|a_i|_\infty
\end{equation}
for all sufficiently large $i$'s
by the above Artin-Schreier equation and $\sum_i a_it^i\in\TT$.
Whence (1) is proved.
(2) follows from \eqref{eq:norm equality}.
(3) is immediate because $\wp$ is $\F_q[t]$-linear.

Suppose that $f$ is in $\E$. Then by \eqref{eq:norm equality},
we see that the inverse image $g$ satisfies the first condition of $\E$.
Put $K'_\infty:=K_\infty(a_0,a_1,a_2,\dots)$.
Then $K'_\infty$ is presented as the field of  Laurent series 
$\F'((\frac{1}{\theta'}))$ with a finite extension $\F'$ of $\F_q$ and an element $\theta'\in\C_\infty$.
Since all the solutions of  the equation \eqref{eq:AS} lie in   $K'_\infty$
whenever $a_i$ lies in a maximal ideal of $K'_\infty$,
we see that $K'_\infty(b_0,b_1,b_2,\dots)$ is a finite extension of
$K_\infty(a_0,a_1,a_2,\dots)$. 
Thus $g$ is in $\E$. (4) is proved.
\end{proof}

\subsection{Carlitz multiple (star) polylogarithms}
\label{sec:CM(S)PL}
We  recall the definition of Carlitz multiple (star) polylogarithms
and also their $t$-motivic variants.

Throughout this paper we fix a $(q-1)$-th root of $-\theta$.
We consider the function
$$
\Omega=
\Omega(t):=(-\theta)^{\frac{-q}{q-1}}\prod_{i=1}^\infty(1-\frac{t}{\theta^{q^i}})
\in \C_\infty[[t]].
$$
It is an entire function, namely it belongs to $\mathcal E$ and $\TT^\times$,
and satisfies the difference equation
\begin{equation*}
\Omega^{(-1)}(t)=(t-\theta)\Omega(t).
\end{equation*}
The value
$$\tilde \pi:=\frac{1}{\Omega(\theta)}$$
is a period of Carlitz module  (cf. \cite{AT90, T-book}).

The {\it Carlitz multiple polylogarithm} (CMPL)
and {\it Carlitz multiple star polylogarithm} (CMSPL),
introduced in \cite{C14, CM},
are defined by the following power series respectively
\begin{equation}\label{eq:CMPL}
\Li_{n_1,\dots,n_d}(z_1,\dots,z_d)=\sum_{0\leqslant i_1<\cdots<i_d}
\frac{z_1^{q^{i_1}}\cdots z_d^{q^{i_d}}}{L_{i_1}^{n_1}\cdots L_{i_d}^{n_d}}
\in \C_\infty[[z_1,\dots,z_d]]
\end{equation}
and
\begin{equation}\label{eq:CMSPL}
\Li^\star_{n_1,\dots,n_d}(z_1,\dots,z_d)=\sum_{0\leqslant i_1\leqslant \cdots \leqslant i_d}
\frac{z_1^{q^{i_1}}\cdots z_d^{q^{i_d}}}{L_{i_1}^{n_1}\cdots L_{i_d}^{n_d}}
\in \C_\infty[[z_1,\dots,z_d]]
\end{equation}
for $n_1,\dots,n_d,d\in\N$,
where
$L_0:=1$ and $L_i:=(\theta-\theta^{q})\cdots (\theta-\theta^{q^i})\in K$ for $i\geqslant 1$.
When $d=1$, they coincide and recover
the  Carlitz polylogarithm of Anderson-Thakur \cite[\S 2.1]{AT90}.
By \cite[\S 5.1]{C14}, CMPL converges in the region 
\begin{equation}\label{eq:D}
\mathbb D=\left\{(z_i)\in \C_\infty^d\bigm|
|z_1/\theta^{\frac{qn_1}{q-1}}|_\infty^{q^{i_1}}\cdots
|z_d/\theta^{\frac{qn_d}{q-1}}|_\infty^{q^{i_d}}\to 0
\text{ as }
0\leqslant i_1<\cdots<i_d\to \infty\right\}
\end{equation}
and CMSPL converges in the similar region $\mathbb D^\star$
replacing $<$ with $\leqslant$,
both of which contain the polydisk
$\mathbb D'=\{(z_i)\in \C_\infty^d\bigm| |z_i|_\infty <q^{\frac{n_iq}{q-1}}\}$. 


For a fixed $d$-tuple of $(Z_1,\dots,Z_d)\in\TT^d$,
the {\it $t$-motivic CMPL} and {\it $t$-motivic CMSPL}
(cf. \cite{CGM})
are defined by the following series respectively
\begin{align}\label{eq:LLii}
\LLii_{n_1,\dots,n_d}(Z_1,\dots,Z_d)
&=\Omega^{-n_1-\cdots-n_d}\sum_{0\leqslant i_1<\cdots<i_d}
(\Omega^{n_1}Z_1)^{(i_1)}\cdots
(\Omega^{n_d}Z_d)^{(i_d)}
\\
&=\sum_{0\leqslant i_1<\cdots<i_d}
\frac{Z_1^{(i_1)}\cdots Z_d^{(i_d)}}
{\mathbb{L}_{i_1}^{n_1}\cdots \mathbb{L}_{i_d}^{n_d}}
, \notag
\end{align}
\begin{align}\label{eq:LLii star}
\LLii_{n_1,\dots,n_d}^\star(Z_1,\dots,Z_d)
&=\Omega^{-n_1-\cdots-n_d}\sum_{0\leqslant i_1\leqslant\cdots\leqslant i_d}
(\Omega^{n_1}Z_1)^{(i_1)}\cdots
(\Omega^{n_d}Z_d)^{(i_d)}
\\
&=\sum_{0\leqslant i_1\leqslant\cdots\leqslant i_d}
\frac{Z_1^{(i_1)}\cdots Z_d^{(i_d)}}
{\mathbb{L}_{i_1}^{n_1}\cdots \mathbb{L}_{i_d}^{n_d}}
, \notag
\end{align}
where
$\mathbb{L}_0=1$ and $\mathbb{L}_i=(t-\theta^{q})\cdots (t-\theta^{q^i})\in K[t]$ for $i\geqslant 1$.
They coincide when $d=1$.
The $t$-motivic CMPL converges with respect to the Gauss norm when
\begin{equation}\label{eq:condition for Z}
(||Z_1||_\infty/|\theta^{\frac{qn_1}{q-1}}|_\infty)^{q^{i_1}}\cdots
(||Z_d||_\infty/|\theta^{\frac{qn_d}{q-1}}|_\infty)^{q^{i_d}}
\to 0
\text{ as }
0\leqslant i_1<\cdots<i_d\to \infty.
\end{equation}
Similarly the $t$-motivic CMSPL converges in the same situation replacing
$<$ with $\leqslant$.
We remind that the substitution $t=\theta$ gives \eqref{eq:CMPL}
and \eqref{eq:CMSPL}.
We have
\begin{equation}\label{eq:recursive for LLii}
\LLii_{n_1,\dots,n_d}(Z_1,\dots,\Omega Z_k,\dots,Z_d)
=\Omega\cdot \LLii_{n_1,\dots,n_k+1,\dots,n_d}(Z_1,\dots,Z_d)
\end{equation}
\begin{equation}\label{eq:recursive for LLii star}
\LLii_{n_1,\dots,n_d}^\star(Z_1,\dots,\Omega Z_k,\dots,Z_d)
=\Omega\cdot \LLii_{n_1,\dots,n_k+1,\dots,n_d}^\star(Z_1,\dots,Z_d)
\end{equation}
for $k$ with $1\leqslant k\leqslant d$ by definition.

\subsection{Continuation of Carlitz polylogarithms}\label{sec:prolong PL}
We explain a method of continuation 
of the Carlitz polylogarithm to $\C_\infty$
and a treatment of branches, 
which  is an initial step for continuation of
the Carlitz multiple polylogarithm (explained in \S \ref{sec:prolong MPL})
and the star version (explained in \S \ref{sec:prolong MSPL}).
Our method consists of three steps.

\subsubsection{Algebraic step}
We introduce the following series
for $Z\in \TT$:
$$
\LLii_{0}(Z)=\sum_{i=0}^\infty{Z^{(i)}}
$$
which is \lq a $(d,n_d)=(1,0)$ version' of \eqref{eq:LLii}.
When $||Z||_\infty<1$, it converges and is $\F_q[t]$-linear with respect to $Z$.
We have
\begin{equation}\label{eq:func eq for Li0}
\LLii_{0}(Z)-\LLii_{0}(Z)^{(1)}=Z,
\end{equation}
that is,
\begin{equation}\label{eq:wp eq for Li0}
\wp(\LLii_{0}(Z))=Z.
\end{equation}
By \eqref{eq:func eq for Li0}
we remark that $\LLii_0(Z)$ converges to an algebraic function when $Z\in\C_\infty$ with $|Z|<1$.
Lemma \ref{lem:AS} enables us to associate each $Z\in\TT$ with 
$\LLii_0(Z)$ in the quotient $\F_q[t]$-module $\TT/ \F_q[t]$ by keeping the above equation, 
which yields the extended $\F_q[t]$-linear map 
$$\vec\LLii_0:\TT\to\TT/ \F_q[t].$$
A {\it branch} $\LLii_0^o:\TT\to \TT$ means an $\F_q$-linear 
lift of $\vec\LLii_0$.

We note that $\vec\LLii_0(Z)$ is congruent to  $\LLii_0^o(Z)$
modulo $\F_q[t]$  when $||Z||_\infty<1$.
By Lemma \ref{lem:AS}, any  $Z\in\TT$ and its any branch $\LLii_0^o(Z)$
have a same radius of convergence.

\subsubsection{Analytic step}
We consider the continuation of the $t$-motivic Carlitz polylogarithm 
by making use of the equality
\begin{equation}\label{eq:LLiin=OmegaLlii0}
\LLii_n(Z)=\Omega^{-n}\LLii_0(\Omega^nZ)
\end{equation}
deduced from \eqref{eq:recursive for LLii}.
\begin{defn}
For $n\in\N$, we define the $\F_q[t]$-linear map
$$
\vec\LLii_n:\TT\to \TT/\Omega^{-n}\F_q[t]
$$
by sending $Z\in\TT$ to
\begin{equation}\label{eq:den LLii}
\vec\LLii_n(Z):=\Omega^{-n}\cdot\vec\LLii_0(\Omega^{n}Z)
\end{equation}
(N.B. $\Omega^{-n}\TT=\TT$).
A {\it branch} $\LLii_n^o:\TT\to \TT$ means an $\F_q$-linear 
lift of $\vec\LLii_n$.
\end{defn}

By \eqref{eq:LLiin=OmegaLlii0},
it is congruent to 
$\LLii_n(Z)$
modulo $\Omega^{-n}\F_q[t]$
when \eqref{eq:condition for Z} holds.
The following properties  will be used later. 

\begin{lem}\label{lem:property of LLii}
Let $n\geqslant 1$ and $Z\in\TT$.
Let $\LLii_n^o(Z)$ be a branch. Then

\rm{(1).} $\Omega^n\cdot\LLii_n^o(Z)\in\E$  when $\Omega^nZ\in \E$.

\rm{(2).} $\Omega^n\LLii_n^o(Z)-(\Omega^n\LLii_n^o(Z))^{(1)}=\Omega^nZ$ when $Z\in\TT$.

\rm{(3).} $\left(\Omega^n\cdot\LLii_n^o(Z)\right)(\theta^{q^k})
=\left(\Omega^n\cdot\LLii^o_n(Z)\right)(\theta)^{q^k}$
for $k\geqslant 1$
when $\Omega^nZ\in\E$.
\end{lem}

\begin{proof}
(1). It follows from Lemma \ref{lem:AS},
\eqref{eq:wp eq for Li0} and \eqref{eq:den LLii}
because we have $\Omega\in\E$.

(2).
Put $\LLii_0^o(\Omega^nZ)=\Omega^n\cdot\LLii_n^o(Z)$.
By \eqref{eq:func eq for Li0}, we have
$$
\LLii_0^o(\Omega^nZ)-\LLii_0^o(\Omega^nZ)^{(1)}=\Omega^nZ,
$$
which implies the equality.

(3)
By Lemma \ref{lem:AS}.(4) and \eqref{eq:func eq for Li0}, we have $\LLii_0^o(\Omega^nZ)\in\E$
when $\Omega^nZ\in\E$.
By evaluating $t=\theta^{q^{h+1}}$ to the above equality, we obtain
$$
\LLii_0^o(\Omega^nZ)(\theta^{q^{h+1}})
-\LLii_0^o(\Omega^nZ)(\theta^{q^h})^{q}=0
$$
for $h\geqslant 0$
because we have $F^{(1)}(t^q)=F(t)^q$ for any $F\in\C_\infty[[t]]$ and
$\Omega(\theta^{q^{h+1}})=0$.
Thus we obtain the formula.
\end{proof}

\subsubsection{Evaluation step}
By the evaluation of $t=\theta$,
we carry out 
the continuation of the Carlitz polylogarithm. 
\begin{defn}
For $n\in\N$, we define the $A$-linear map
$$\vec\Li_n:\C_\infty\to \C_\infty/\tilde\pi^nA$$ 
by a restriction of $\vec\LLii_n$ to
 $Z=z\in\C_\infty\subset \TT$ and a substitution of $t=\theta$ there
(we note that $t=\theta$ is inside a region of convergence of $\vec\LLii_n(z)$
because $\Omega^n\LLii_n^o(Z)$ and $\Omega^nZ$ have a same radius of convergence
by Lemma \ref{lem:AS}.(2)).
A {\it branch} $\Li_n^o:\C_\infty\to \C_\infty$ means an $\F_q$-linear lift of
$\vec\Li_n$.
\end{defn}

The following proposition ensures that $\vec\Li_n$ is an analytic continuation of $\Li_n$.

\begin{prop}\label{prop:ana con CPL}
{\rm (1).} $\vec\Li_n(z)\equiv\Li_n(z) \bmod \tilde\pi^nA$
when $z$ lies in $\mathbb D$.

{\rm (2).}
$\vec\Li_n$ locally admits an analytic  lift,  that is, for each
given  point $z\in \C_\infty$, there is  a small disk $U$ centered at $z$
in $\C_\infty$ and a lift $\vec\Li_n|_U:U \to \C_\infty$
which is given by a converging  power series.
\end{prop}

\begin{proof}
{\rm (1).}
It follows from our construction.

{\rm (2).}
By our construction we have
$\vec\Li_n(z+w)\equiv\vec\Li_n(z)+\vec\Li_n(w)$.
We have $\vec\Li_n(w)\equiv\Li_n(w)$ for $w\in \mathbb{D}$ and $\Li_n(w)$ is rigid analytic on an appropriately smaller closed disk centered at $0$.
Then our claim follows because 
$\tilde\pi^nA$ is discrete in $\C_\infty$.
\end{proof}

\begin{rem}
By definition, difference of any two branches of the Carlitz polylogarithm
$\Li_n(z)$ is given by
$$\alpha\cdot\tilde\pi^n \qquad (\alpha\in A).$$
While it is worthy to recall
in the complex case (characteristic $0$ case),
difference of any two branches of
(analytically continued) polylogarithm  $\Li_n(z)$
is given by  a
$\Q$-linear combination of
 $$(2\pi i)^a\zeta(b)(\log z)^c$$
 with $a+b+c=n$.
\end{rem}

Here is an application of our method to evaluate an analytic continuation 
at a point  outside of the region of convergence of $\Li_n(z)$:

\begin{rem}
Let $n\in\N$ and $\alpha\in\C_\infty$.
Let $\sum_{i=0}^\infty a_it^i$ be the power series expansion of 
$\alpha\Omega^n\in \C_\infty[[t]]$.
Take a series $\{b_i\}_{i=0}^\infty$ such that
$\wp(b_i)=a_i$ for all $i$ 
and $|b_i|_\infty=|a_i|_\infty$ for all sufficiently large $i$'s.
Then  $\tilde\pi^n\sum_{i=0}^\infty b_i\theta^i\in\C_\infty$ gives a representative of lifts  of 
$\vec\Li_n(\alpha)\in\C_\infty/\tilde\pi^nA$.
\end{rem}

\subsection{Continuation of Carlitz multiple polylogarithms}\label{sec:prolong MPL}
By exploiting the method of continuation of Carlitz polylogarithm
developed in \S \ref{sec:prolong PL},
we extend the Carlitz multiple polylogarithm to $\C_\infty$
with a treatment of branches, that is, a monodromy module
by three steps in a similar fashion.

\subsubsection{Algebraic step}
We denote $\{0\}^d$ to be the multi-index where $0$ is repeated $d$-times
and consider the following series
for $Z_1,\dots, Z_d\in \TT$:
$$
\LLii_{\{0\}^d}(Z_1,\dots,Z_d)=\sum_{0\leqslant i_1<\cdots<i_d}
{Z_1^{(i_1)}}\cdots{Z_d^{(i_d)}}
$$
which is  \lq an $(n_1,\dots,n_d)=\{0\}^d$ version' of \eqref{eq:LLii}.
When \eqref{eq:condition for Z} holds for $(n_1,\dots,n_d)=\{0\}^{d}$,
it converges and  is $\F_q[t]$-linear with respect to 
$Z_1,\dots,Z_d$.
We observe the  following system of difference equations
\begin{equation}\label{eq:sys diff eq}
\wp(\LLii_{\{0\}^i}(Z_{d-i+1},\dots,Z_d))=
Z_{d-i+1}\cdot\LLii_{\{0\}^{i-1}}(Z_{d-i+2},\dots,Z_d)^{(1)}
\end{equation}
which they satisfy for $1\leqslant i\leqslant d$.
Here we put $\LLii_{\{0\}^0}=1$.
Again by \eqref{eq:sys diff eq}
we remark that $\LLii_{\{0\}^d}(Z_1,\dots,Z_d)$ converges to an algebraic function when $Z_1,\dots, Z_d$ are in $\C_\infty$ and $|Z_1|_\infty,\dots, |Z_d|_\infty$
are enough small.

We note that, by Lemma \ref{lem:AS}.(1), for any $Z_1,\dots, Z_d\in \TT$,
there always exists a solution of the above system 
\eqref{eq:sys diff eq}, 
denoted by
\begin{align}\label{eqLio0dZ}
\vec\LLii_{\{0\}^{d}}^o( & Z_1,\dots, Z_d):= \\ \notag
&(
\LLii_{\{0\}^{1}}^o(Z_d),
\dots,
\LLii_{\{0\}^{d-1}}^o(Z_2,\dots, Z_d),
\LLii_{\{0\}^{d}}^o(Z_1,\dots, Z_d)
)^\T
\in\TT^d,
\end{align}
and all the  solutions of the above system \eqref{eq:sys diff eq}
are described as  linear combinations
\begin{equation}\label{eq:solution linear combination}
\vec\LLii_{\{0\}^{d}}^o(Z_1,\dots, Z_d)+\sum_{k=0}^{d-1}
\alpha_k\cdot\vec\LLii_{\{0\}^{d}}^o(Z_1,\dots, Z_k)
\end{equation}
with $\alpha_k\in\F_q[t]$ and
$$
\vec\LLii_{\{0\}^{d}}^o(Z_1,\dots, Z_k)
:=
(
\{0\}^{d-k-1}, 1,
\LLii_{\{0\}^{1}}^o(Z_k),
\dots,
\LLii_{\{0\}^{k-1}}^o(Z_2,\dots, Z_k), 
\LLii_{\{0\}^{k}}^o(Z_1,\dots, Z_k)
)^\T
$$ 
in $\TT^d$
whose last $k$ components are
solutions of \eqref{eq:sys diff eq} with $d=k$.
When $k=0$, it means $(0,\dots,0,1)^\T$.

We put 
$\M_{\{0\}^d}^{Z_1,\dots,Z_{d-1}}$
to be the $\F_q[t]$-submodule of $\TT^d$ generated by the $d$ elements:
$$
\M_{\{0\}^d}^{Z_1,\dots,Z_{d-1}}:=\langle
\vec\LLii_{\{0\}^{d}}^o(Z_1,\dots, Z_k) \bigm| 0\leqslant k\leqslant d-1
\rangle_{\F_q[t]}.
$$
It follows from \eqref{eq:solution linear combination} that 
$\M_{\{0\}^d}^{Z_1,\dots,Z_{d-1}}$
is free from any choice of branches.
For fixed $Z_1,\dots, Z_{d-1}\in\TT$,
we obtain a well-defined $\F_q[t]$-linear map
$$
\vec\LLii_{\{0\}^{d}}(Z_1,\dots, Z_{d-1},-):\TT\to
\TT^d/\M_{\{0\}^d}^{Z_1,\dots,Z_{d-1}}.
$$
A {\it branch} $\vec\LLii_{\{0\}^{d}}^o(Z_1,\dots, Z_{d-1},-):\TT\to \TT^d$
means an $\F_q$-linear lift of $\vec\LLii_{\{0\}^{d}}(Z_1,\dots, Z_{d-1},-)$.

We note that the vector \eqref{eqLio0dZ} is congruent to its \lq non-$o$' version
$$
(
\{0\}^{d-k-1}, 1,
\LLii_{\{0\}^{1}}(Z_k),
\dots,
\LLii_{\{0\}^{k}}(Z_1,\dots, Z_k)
)^\T
$$
modulo $\M_{\{0\}^d}^{Z_1,\dots,Z_{d-1}}$ 
when all components converge.

\subsubsection{Analytic step}
For $n_1,\dots,n_d\geqslant 1$, $Z_1,\dots,Z_d\in\TT$,
we put 
$$
\M_{n_1,\dots,n_d}^{Z_1,\dots,Z_{d-1}}:=
\Omega^{-n_1-\cdots-n_d}\M_{\{0\}^d}^{\Omega^{n_1}Z_1,\dots,\Omega^{n_{d-1}}Z_{d-1}}
$$
which is an $\F_q[t]$-submodule of $\TT^d$ because $\Omega^{-1}\in\TT$.
Then the continuation of the $t$-motivic Carlitz multiple polylogarithm 
is carried out as follows:

\begin{defn}\label{def:multiple LLii}
Let $n_1,\dots,n_d\in\N$.
For fixed $Z_1,\dots, Z_{d-1}\in\TT$, we define the $\F_q[t]$-linear map
$$
\vec\LLii_{n_1,\dots,n_d}(Z_1,\dots,Z_{d-1},-): \TT\to\TT^d/\M_{n_1,\dots,n_d}^{Z_1,\dots,Z_{d-1}}
$$
sending $Z_d\in\TT$ to
\begin{align*}\label{eq:den multiple LLii}
\vec\LLii_{n_1,\dots,n_d} (Z_1,\dots,Z_d)
:=\Omega^{-n_1-\cdots-n_d}
\vec\LLii_{\{0\}^d}(\Omega^{n_1}Z_1,\dots,\Omega^{n_d}Z_d).
\end{align*}
A {\it branch}
$\vec\LLii_{n_1,\dots,n_d}^o(Z_1,\dots,Z_{d-1},-):\TT\to\TT^d$
means an $\F_q$-linear lift of $\vec\LLii_{n_1,\dots,n_d}^o(Z_1,\dots,Z_{d-1},-)$
and, 
for each $Z_d\in\TT$, 
we denote
\begin{align*}
\vec\LLii_{n_1,\dots,n_d}^o(Z_1,\dots,Z_{d})= 
\bigl( &
\Omega^{-n_1-\dots-n_{d-1}}\LLii_{n_d}^o(Z_d),
\dots,
\Omega^{-n_1}\LLii_{n_2,\dots,n_d}^o (Z_2,\dots,Z_d),\\
&\qquad
\LLii_{n_1,\dots,n_d}^o (Z_1,\dots,Z_d)
\bigr)^\T\in\TT^d.
\end{align*}
\end{defn}

It turns that 
the module
$\M_{n_1,\dots,n_d}^{Z_1,\dots,Z_{d-1}}$
is 
generated by $d$ elements, in precise,
$$
\M_{n_1,\dots,n_d}^{Z_1,\dots,Z_{d-1}}
=\langle
\vec\LLii_{n_1,\dots,n_d}^o(Z_1,\dots, Z_k) 
\bigm| 0\leqslant k\leqslant d-1
\rangle_{\F_q[t]},
$$
with
\begin{align*}
\vec\LLii_{n_1,\dots,n_d}^o&(Z_1, \dots, Z_k)
:= 
\Omega^{-n_1-\cdots-n_d}\cdot 
\vec\LLii_{\{0\}^{d}}^o(\Omega^{n_1}Z_1,\dots, \Omega^{n_k}Z_k)
  \\
&=\Omega^{-n_{k+1}-\cdots-n_d}\cdot
\Bigl(
\{0\}^{d-k-1},
\Omega^{-n_1-\dots-n_{k}},
\Omega^{-n_1-\dots-n_{k-1}}\LLii_{n_k}^o(Z_k), \\
&\qquad\qquad\dots,
 \Omega^{-n_1}\LLii_{n_2,\dots,n_k}^o (Z_2,\dots,Z_k), 
\LLii_{n_1,\dots,n_k}^o (Z_1,\dots,Z_k)
\Bigr)^\T
\in\TT^d
\end{align*}
with $0\leqslant k\leqslant d-1$.
The definition of
$\M_{n_1,\dots,n_d}^{Z_1,\dots,Z_{d-1}}$
is independent of any choice of branches.
 
Again we note that the vector $\vec\LLii_{n_1,\dots,n_d}^o(Z_1, \dots, Z_d)$ is congruent to its \lq non-$o$' version
modulo $\M_{n_1,\dots,n_d}^{Z_1,\dots,Z_{d-1}}$
when all components converge. 

%
The following properties  will be used in our later sections.

\begin{prop}\label{lem:property of multiple LLii}
Put $n_1,\dots, n_d\geqslant 1$ and $Z_1,\dots,Z_d\in\TT$.
Let 
$
\vec\LLii_{n_1,\dots,n_d}^o(Z_1,\dots,Z_d)$
be a branch as above.
Then we have

\rm{(1).}
A congruence with
the tuple 
$$
(
\Omega^{-n_1-\dots-n_{d-1}}\LLii_{n_d}(Z_d),
\dots,
\Omega^{-n_1}\LLii_{n_2,\dots,n_d} (Z_2,\dots,Z_d),
\LLii_{n_1,\dots,n_d} (Z_1,\dots,Z_d)
\bigr)^\T
\in \TT^d
$$
of \eqref{eq:LLii} modulo $\M_{n_1,\dots,n_d}^{Z_1,\dots,Z_{d-1}}$
when \eqref{eq:condition for Z} holds.

\rm{(2).} $\Omega^{n_1+\cdots+n_d}\vec\LLii_{n_1,\dots,n_d}^o(Z_1,\dots,Z_k)\in\TT(\infty)^d$ (resp. $\E^d$)
for $k=1,\dots,d$
when $\Omega^{n_1}Z_1,\dots,\Omega^{n_d}Z_d\in\TT(\infty)$ (resp. $\E$).

\rm{(3).} 
$
\wp\left(\Omega^{n_1+\cdots+n_d} \LLii_{n_1,\dots,n_d}^o(Z_1,\dots,Z_d)\right)
=
\Omega^{n_1}Z_1
\left(\Omega^{n_2+\cdots+n_d}\LLii_{n_2,\dots,n_d}^o(Z_2,\dots,Z_d)\right)^{(1)}.
$

\rm{(4).} $\left(\Omega^{n_1+\cdots+n_d}\LLii_{n_1,\dots,n_d}^o(Z_1,\dots,Z_d)\right)(\theta^{q^k})
=\left(\Omega^{n_1+\cdots+n_d}\LLii_{n_1,\dots,n_d}^o(Z_1,\dots,Z_d)\right)(\theta)^{q^k}$
for $k\geqslant 1$
when $\Omega^{n_1}Z_1,\dots,\Omega^{n_d}Z_d\in\E$.
\end{prop}

\begin{proof}
The proof  can be done in the same way to that of 
Lemma \ref{lem:property of LLii}.

(1).
It can be deduced from \eqref{eq:recursive for LLii}.

(2). 
By Lemma \ref{lem:AS}, we have
$\vec\LLii_{\{0\}^d}^o(Z_1,\dots,Z_{d})\in\TT(\infty)^d$ (resp. $\E^d$)
for $Z_1,\dots,Z_{d}\in\TT(\infty)$ (resp. $\E$),
which implies the claim.

(3).
Put
$$
\LLii_{\{0\}^d}^o(\Omega^{n_1}Z_1,\dots, \Omega^{n_d}Z_d)
=\Omega^{n_1+\dots+n_d}\cdot\LLii_{n_1,\dots,n_d}^o(Z_1,\dots,Z_d).
$$
By \eqref{eq:sys diff eq}, we have
\begin{align*}
\LLii_{\{0\}^d}^o(\Omega^{n_1}Z_1,\dots, \Omega^{n_d}Z_d)-&
\LLii_{\{0\}^d}^o(\Omega^{n_1}Z_1,\dots, \Omega^{n_d}Z_d)^{(1)} \\
&=\Omega^{n_1}Z_1\LLii_{\{0\}^{d-1}}^o(\Omega^{n_2}Z_2,\dots, \Omega^{n_d}Z_d)^{(1)}
\end{align*}
which proves the claim.

(4).  By Lemma \ref{lem:AS} and \eqref{eq:sys diff eq}, we inductively obtain
$
\LLii_{\{0\}^d}^o(\Omega^{n_1}Z_1,\dots, \Omega^{n_d}Z_d)\in\E
$.
Evaluation of $t=\theta^{q^{h+1}}$ to  the above equation yields
$$
\LLii_{\{0\}^d}^o(\Omega^{n_1}Z_1,\dots, \Omega^{n_d}Z_d)(\theta^{q^{h+1}})-
\LLii_{\{0\}^d}^o(\Omega^{n_1}Z_1,\dots, \Omega^{n_d}Z_d)(\theta^{q^h})^q
=0
$$
by the same reason to the proof of Lemma \ref{lem:property of LLii}.
\end{proof}

We remark that (2) and (4) are shown in \cite[Lemma 5.3.1 and 5.3.5]{C14}
under the  convergence condition for $(Z_1,\dots,Z_d)\in (\bar K^\times)^d\cap \mathbb D$
and 
in  \cite[Proposition 2.3.3]{CPY} under the condition
$(Z_1,\dots,Z_d)\in (\bar K[t])^d\cap \mathbb D$.

\subsubsection{Evaluation step}
By the evaluation of $t=\theta$,
we carry out 
the continuation of the Carlitz multiple polylogarithm.

\begin{defn}
Let $n_1,\dots,n_d\in\N$
and $z_1,\dots, z_{d-1}\in\C_\infty$.

(1). 
The {\it monodromy module}
$$M_{n_1,\dots,n_d}^{z_1,\dots,z_{d-1}}$$
is defined to be the $\F_q[t]$-submodule of $\C_\infty^d$
given by the evaluation of $t=\theta$ to
$\M_{n_1,\dots,n_d}^{z_1,\dots,z_{d-1}}$.

(2).
We define
the $\F_q$-linear map
$$
\vec\Li_{n_1,\dots,n_d}(z_1,\dots,z_{d-1},-):\C_\infty\to \C_\infty^d/M_{n_1,\dots,n_d}^{z_1,\dots,z_{d-1}}$$ 
by a restriction of
$\vec\LLii_{n_1,\dots,n_d}(Z_1,\dots,Z_{d-1},-)$
to  $Z_i=z_i\in\C_\infty\subset \TT$ and a substitution of $t=\theta$ there 
(we note again that $t=\theta$ is inside its region of convergence
by $\Omega(\theta)\neq 0$,
the entireness of  $\Omega$ and the above proposition).

(3).
A {\it branch} $\vec\Li_{n_1,\dots,n_d}^o(z_1,\dots,z_{d-1},-)$
means an $\F_q$-linear lift of
$\vec\Li_{n_1,\dots,n_d}(z_1,\dots,z_{d-1},-)$.
For each $z_d\in\C_\infty$,
we denote 
\begin{align*}
\vec\Li_{n_1,\dots,n_d}^o(z_1,\dots,z_{d})= &
(
\tilde\pi^{n_{1}+\cdots+n_{d-1}}\Li_{n_d}^o(z_d),
\tilde\pi^{n_{1}+\cdots+n_{d-2}}\Li_{n_{d-1},n_d}^o(z_{d-1},z_d),
\dots, \\
&\qquad
\tilde\pi^{n_{1}}\Li_{n_2,\dots,n_d}^o(z_2,\dots, z_d), 
\Li_{n_1,\dots,n_d}^o(z_1,\dots, z_d)
)^\T
\in \C_\infty^d.
\end{align*}
\end{defn}

The definition of the monodromy module $M_{n_1,\dots,n_d}^{z_1,\dots,z_{d-1}}$ is
independent of any choice of branches.
It is 
the $A$-submodule of $\C_\infty^d$
generated by $d$ elements, in precise,
$$
M_{n_1,\dots,n_d}^{z_1,\dots,z_{d-1}}
=\langle
\vec\Li_{n_1,\dots,n_d}^o(z_1,\dots, z_k) 
\bigm| 0\leqslant k\leqslant d-1
\rangle_{A},
$$
with
\begin{align*}
&\vec\Li_{n_1,\dots,n_d}^o(z_1, \dots, z_k)
:= 
\vec\LLii_{n_1,\dots,n_d}^o(z_1,\dots,z_k)|_{t=\theta}
  \\
&\quad=\tilde\pi^{n_{k+1}+\cdots+n_d}\cdot
\Bigl(
\{0\}^{d-k-1},
\tilde\pi^{n_1+\dots+n_{k}},
\tilde\pi^{n_1+\dots+n_{k-1}}\Li_{n_k}^o(z_k), 
\dots, \\
&
\qquad\qquad\qquad\qquad
\tilde\pi^{n_{1}}\Li_{n_2,\dots,n_d}^o(z_2,\dots, z_d), 
\Li_{n_1,\dots,n_k}^o (z_1,\dots,z_k)
\Bigr)^\T\in\C_\infty^d
\end{align*}
with $0\leqslant k\leqslant d-1$.
In other word, it is the $A$-submodule of $\C_\infty^d$
generated by $d$ columns of the following matrix:
$$
{\tiny
\begin{pmatrix}
      \tilde\pi^{n_1+\dots+n_{d}} & 0 & 0 & \ldots & 0 \\
      \tilde\pi^{n_1+\dots+n_{d-2}+n_d}\Li_{n_{d-1}}^o(z_{d-1})   & \tilde\pi^{n_1+\dots+n_{d}} & 0 & \dots & \vdots \\
      \vdots & \vdots & \ddots &0 &  \vdots\\
      \vdots   & \vdots & & \tilde\pi^{n_1+\cdots+n_d} & 0 \\
      \tilde\pi^{n_d}\Li_{n_1,\dots,n_{d-1}}^o(z_1,\dots,z_{d-1})  &  \tilde\pi^{n_{d-1}+n_d}\Li_{n_1,\dots,n_{d-2}}^o(z_1,\dots,z_{d-2}) & \dots &  \tilde\pi^{n_2+\cdots+n_d}\Li_{n_1}^o(z_1)  &  \tilde\pi^{n_1+\cdots+n_d}
\end{pmatrix}.
}
$$ 
%


\begin{thm}\label{prop:ana con CMPL}
{\rm (1).} $\vec\Li_{n_1,\dots,n_d}(z_1,\dots,z_{d})$
is congruent to the tuple
$$
(
\tilde\pi^{n_{1}+\cdots+n_{d-1}}\Li_{n_d}(z_d),
\dots,
\tilde\pi^{n_{1}}\Li_{n_2,\dots,n_d}(z_2,\dots, z_d), 
\Li_{n_1,\dots,n_d}(z_1,\dots, z_d)
)^\T
\in \C_\infty^d
$$
of \eqref{eq:CMPL}
modulo $M_{n_1,\dots,n_d}^{z_1,\dots,z_{d-1}}$
when $(z_1,\dots,z_d)$ lies in $\mathbb D$.

{\rm (2).}
$\vec\Li_{n_1,\dots,n_d}(z_1,\dots,z_{d-1},-)$ 
locally admits an analytic  lift
(as a function on $z_d$)
in the sense of Proposition \ref{prop:ana con CPL}.
\end{thm}



\begin{proof}
The proof can be done in the same way to that of Proposition \ref{prop:ana con CPL}.
We note that 
$M_{n_1,\dots,n_d}^{z_1,\dots,z_{d-1}}$ is discrete in $\C_\infty^d$
because 
the above matrix
forms a lower triangular matrix with invertible diagonals. 
\end{proof}

\begin{rem}\label{rem:other branch of CMPL}
Let 
$\vec\Li^o_{n_1,\dots,n_d}(z_1,\dots,z_{d})$ and
$\vec\Li^{o'}_{n_1,\dots,n_d}(z_1,\dots,z_{d})$ 
be any branches and denote their last coordinates by
$\Li^o_{n_1,\dots,n_d}(z_1,\dots,z_{d})$ and
$\Li^{o'}_{n_1,\dots,n_d}(z_1,\dots,z_{d})$
respectively. 
By definition, the difference  between them
is given by  an integral combination of the last row of the above matrix:
$$
\Li^o_{n_1,\dots,n_d}(z_1,\dots,z_{d})-
\Li^{o'}_{n_1,\dots,n_d}(z_1,\dots,z_{d})=
\sum_{i=0}^{d-1}
\alpha_i\cdot\tilde\pi^{n_{i+1}+\cdots+n_d}\Li_{n_1,\dots,n_i}^o(z_1,\dots,z_i)
$$
with $\alpha_i\in A$.
\end{rem}


\subsection{Continuation of Carlitz multiple star polylogarithms}\label{sec:prolong MSPL}
By exploiting the method of continuation of Carlitz polylogarithm
in \S \ref{sec:prolong PL}
and imitating the arguments in  \S \ref{sec:prolong MPL},
we extend the Carlitz multiple star polylogarithm to $\C_\infty$
with a treatment of branches, that is, a monodromy module.

\subsubsection{Algebraic step}
We consider the series
for $Z_1,\dots, Z_d\in \TT$:
$$
\LLii_{\{0\}^d}^\star(Z_1,\dots,Z_d)=\sum_{0\leqslant i_1\leqslant \cdots\leqslant i_d}
{Z_1^{(i_1)}}\cdots{Z_d^{(i_d)}}.
$$
We observe the  following system of difference equations
\begin{equation}\label{eq:sys diff eq star}
\wp(\LLii_{\{0\}^i}^\star(Z_{d-i+1},\dots,Z_d))=
Z_{d-i+1}\cdot\LLii_{\{0\}^{i-1}}^\star(Z_{d-i+2},\dots,Z_d)
\end{equation}
which they satisfy for $1\leqslant i\leqslant d$.

We note that, by Lemma \ref{lem:AS}.(1), for any $Z_1,\dots, Z_d\in \TT$,
there always exists a solution of the above system 
\eqref{eq:sys diff eq star} in $\TT^d$,
denoted by
\footnote{
For our convenience in the next section,
we reverse here the order of  coordinate 
to that of 
$\vec\LLii_{\{0\}^{d}}^{o}(Z_1,\dots, Z_d)$ in the previous subsection.
}
\begin{align}\label{eqLio0stardZ}
\vec\LLii_{\{0\}^{d}}^{\star,o}&(Z_1,\dots, Z_d):= \\ \notag
(
&\LLii_{\{0\}^{d}}^{\star,o}(Z_1,\dots, Z_d),
\LLii_{\{0\}^{d-1}}^{\star,o}(Z_2,\dots, Z_d),\dots,
\LLii_{\{0\}^{1}}^{\star,o}(Z_d)
)^\T,
\end{align}
and all the  solutions of the above system \eqref{eq:sys diff eq star}
are described as linear combinations
$$
\vec\LLii_{\{0\}^{d}}^{\star,o}(Z_1,\dots, Z_d)+\sum_{k=0}^{d-1}
\alpha_k\cdot\vec\LLii_{\{0\}^{d}}^{\star,o}(Z_1,\dots, Z_k)
$$
with $\alpha_k\in\F_q[t]$ and
$$
\vec\LLii_{\{0\}^{d}}^{\star,o}(Z_1,\dots, Z_k)
:=
(
\LLii_{\{0\}^{k}}^{\star,o}(Z_1,\dots, Z_k),
\LLii_{\{0\}^{k-1}}^{\star,o}(Z_2,\dots, Z_k), 
\dots,
\LLii_{\{0\}^{1}}^{\star,o}(Z_k),
1,
\{0\}^{d-k-1}
)^\T
$$ 
in $\TT^d$
whose first $k$ components are
solutions of \eqref{eq:sys diff eq star} with $d=k$.
When $k=0$, it means $(0,\dots,0,1)^\T$.

Put 
$\M_{\{0\}^d}^{\star,Z_1,\dots,Z_{d-1}}$
to be the $\F_q[t]$-submodule of $\TT^d$ generated by the $d$ elements: 
$$
\M_{\{0\}^d}^{\star,Z_1,\dots,Z_{d-1}}:=\langle
\vec\LLii_{\{0\}^{d}}^{\star,o}(Z_1,\dots, Z_k) \bigm| 0\leqslant k\leqslant d-1
\rangle_{\F_q[t]},
$$
which is actually independent of any choice of branches.
Then for a fixed $Z_1,\dots, Z_{d-1}\in\TT$,
we obtain a well-defined $\F_q[t]$-linear map
$$
\vec\LLii_{\{0\}^{d}}^{\star}(Z_1,\dots, Z_{d-1},-):\TT\to
\TT^d/\M_{\{0\}^d}^{\star,Z_1,\dots,Z_{d-1}}.
$$
A {\it branch} $\vec\LLii_{\{0\}^{d}}^{\star,o}(Z_1,\dots, Z_{d-1},-):\TT\to\TT^d$
means an $\F_q$-linear lift of $\vec\LLii_{\{0\}^{d}}^{\star}(Z_1,\dots, Z_{d-1},-)$.

We note that the vector \eqref{eqLio0stardZ} is congruent to its \lq non-$o$' version
modulo $\M_{\{0\}^d}^{\star,Z_1,\dots,Z_{d-1}}$ 
when all components converge.

\subsubsection{Analytic step}
For $n_1,\dots,n_d\geqslant 1$, $Z_1,\dots,Z_d\in\TT$,
we put 
$$
\M_{n_1,\dots,n_d}^{\star, Z_1,\dots,Z_{d-1}}:=
\Omega^{-n_1-\cdots-n_d}\M_{\{0\}^d}^{\star,\Omega^{n_1}Z_1,\dots,\Omega^{n_{d-1}}Z_{d-1}},
$$
which is the $\F_q[t]$-submodule of $\TT^d$.
Then the continuation of the $t$-motivic Carlitz star multiple polylogarithm 
is carried out  as follows:

\begin{defn}\label{def:multiple LLii star}
Let $n_1,\dots,n_d\in\N$.
For fixed $Z_1,\dots, Z_{d-1}\in\TT$, we define 
the $\F_q[t]$-linear map
$$
\vec\LLii_{n_1,\dots,n_d}^{\star}(Z_1,\dots,Z_{d-1},-): \TT\to\TT^d/\M_{n_1,\dots,n_d}^{\star,Z_1,\dots,Z_{d-1}}
$$
sending $Z_d\in\TT$ to
\begin{align*}\label{eq:den multiple LLii star}
\vec\LLii_{n_1,\dots,n_d}^{\star} (Z_1,\dots,Z_d)
:=\Omega^{-n_1-\cdots-n_d}
\vec\LLii_{\{0\}^d}^{\star}(\Omega^{n_1}Z_1,\dots,\Omega^{n_d}Z_d).
\end{align*}
A {\it branch}  
$\vec\LLii_{n_1,\dots,n_d}^{\star,o}(Z_1,\dots,Z_{d-1},-):\TT\to \TT^d$
means an $\F_q$-linear lift of
$\vec\LLii_{n_1,\dots,n_d}^{\star}(Z_1,\dots,Z_{d-1},-)$
and, for each $Z_d\in\TT$, we denote 
\begin{align*}
\vec\LLii_{n_1,\dots,n_d}^{\star,o}(Z_1,\dots,Z_{d})= 
&\bigl(
\LLii_{n_1,\dots,n_d}^{\star,o} (Z_1,\dots,Z_d),
\Omega^{-n_1}\LLii_{n_2,\dots,n_d}^{\star,o} (Z_2,\dots,Z_d),\\
&
\dots,
\Omega^{-n_1-\dots-n_{d-1}}\LLii_{n_d}^{\star,o}(Z_d)
\bigr)^\T\in\TT^d.
\end{align*}
\end{defn}

It turns out that 
the module
$\M_{n_1,\dots,n_d}^{\star,Z_1,\dots,Z_{d-1}}$
is the $\F_q[t]$-submodule of $\TT^d$
generated by $d$ elements, in precise,
$$
\M_{n_1,\dots,n_d}^{\star,Z_1,\dots,Z_{d-1}}=\langle
\vec\LLii_{n_1,\dots,n_d}^{\star,o}(Z_1,\dots, Z_k) \bigm| 0\leqslant k\leqslant d-1
\rangle_{\F_q[t]}.
$$
with
\begin{align*}
\vec\LLii_{n_1,\dots,n_d}^{\star,o}&(Z_1,\dots, Z_k)
:=\Omega^{-n_1-\cdots-n_d}\cdot
\vec\LLii_{\{0\}^{d}}^{\star,o}(\Omega^{n_1}Z_1,\dots, \Omega^{n_k}Z_k) \\
 &
=\Omega^{-n_{k+1}-\cdots-n_d}\cdot
\Bigl(
\LLii_{n_1,\dots,n_k}^{\star,o} (Z_1,\dots,Z_k), 
 \Omega^{-n_1}\LLii_{n_2,\dots,n_k}^{\star,o} (Z_2,\dots,Z_k), \\
&\qquad\qquad\quad
\dots, 
\Omega^{-n_1-\dots-n_{k-1}}\LLii_{n_k}^{\star,o}(Z_k), 
\Omega^{-n_1-\dots-n_{k}},
\{0\}^{d-k-1}
\Bigr)^\T\in\TT^d
\end{align*}
with $0\leqslant k\leqslant d-1$.
Again the definition of $\M_{n_1,\dots,n_d}^{\star,Z_1,\dots,Z_{d-1}}$
is independent of any branches.
The following properties  will be used in our later sections.

\begin{prop}\label{lem:property of multiple LLii star}
Put $n_1,\dots, n_d\geqslant 1$ and $Z_1,\dots,Z_d\in\TT$.
Let 
$
\vec\LLii_{n_1,\dots,n_d}^{\star,o}(Z_1,\dots,Z_d)$
be a branch as above.
Then we have

\rm{(1).}
A congruence with
the tuple 
given by
$$
(
\LLii_{n_1,\dots,n_d}^{\star}(Z_1,\dots,Z_d),
\Omega^{-n_1}\LLii_{n_2,\dots,n_d}^{\star} (Z_2,\dots,Z_d),
\dots,
\Omega^{-n_1-\dots-n_{d-1}}\LLii_{n_d}^{\star}(Z_d)
\bigr)^\T \in \TT^d
$$
of \eqref{eq:LLii star} modulo $\M_{n_1,\dots,n_d}^{\star,Z_1,\dots,Z_{d-1}}$
when it converges.

\rm{(2).} $\Omega^{n_1+\cdots+n_d}\vec\LLii_{n_1,\dots,n_d}^{\star,o}(Z_1,\dots,Z_k)\in\TT(\infty)^d$ (resp. $\E^d$) for $k=1,\dots,d$
when $\Omega^{n_1}Z_1,\dots,\Omega^{n_d}Z_d\in\TT(\infty)$ (resp. $\E$).

\rm{(3).} 
$
\wp\left(\Omega^{n_1+\cdots+n_d} \LLii_{n_1,\dots,n_d}^{\star,o}(Z_1,\dots,Z_d)\right)
=\Omega^{n_1}Z_1\cdot
\Omega^{n_2+\cdots+n_d}\LLii_{n_2,\dots,n_d}^{\star,o}(Z_2,\dots,Z_d).
$

\rm{(4).} $\left(\Omega^{n_1+\cdots+n_d}\LLii_{n_1,\dots,n_d}^{\star,o}(Z_1,\dots,Z_d)\right)(\theta^{q^k})
=\left(\Omega^{n_1+\cdots+n_d}\LLii_{n_1,\dots,n_d}^{\star,o}(Z_1,\dots,Z_d)\right)(\theta)^{q^k}$
for $k\geqslant 1$
when $\Omega^{n_1}Z_1,\dots,\Omega^{n_d}Z_d\in\E$.
\end{prop}

\begin{proof}
The proof  can be done in the same way to that of 
Proposition \ref{lem:property of multiple LLii}.
%
%
%
%
\end{proof}

\subsubsection{Evaluation step}
By the evaluation of $t=\theta$, we carry out the continuation of the
Carlitz star multiple polylogarithm.

\begin{defn}
Let $n_1,\dots,n_d\in\N$
and $z_1,\dots, z_{d-1}\in\C_\infty$.

(1). The {\it monodromy module} $M_{n_1,\dots,n_d}^{\star, z_1,\dots,z_{d-1}}$
is defined to be the $\F_q[t]$-submodule of $\C_\infty^d$
given by the evaluation of $t=\theta$ to $\M_{n_1,\dots,n_d}^{\star, z_1,\dots,z_{d-1}}$. 

(2). We define the $\F_q$-linear map
$$
\vec\Li_{n_1,\dots,n_d}^\star(z_1,\dots,z_{d-1},-):\C_\infty\to \C_\infty^d/M_{n_1,\dots,n_d}^{\star, z_1,\dots,z_{d-1}}$$ 
by a restriction of
$\vec\LLii_{n_1,\dots,n_d}^\star(Z_1,\dots,Z_{d-1},-)$
to $Z_i=z_i\in\C_\infty\subset \TT$ and a substitution of $t=\theta$ there
(we note again that $t=\theta$ is inside its region of convergence
by $\Omega(\theta)\neq 0$,
the entireness of  $\Omega$ and the above proposition).

(3). A {\it branch} 
$\vec\LLii_{n_1,\dots,n_d}^{\star,o}(Z_1,\dots,Z_{d-1},-):\C_\infty\to \C_\infty^d $
means an $\F_q$-linear lift of 
$\vec\Li_{n_1,\dots,n_d}^\star(z_1,\dots,z_{d})$
and, for each $z_d\in\C_\infty$, we denote
\begin{align*}
\vec\Li_{n_1,\dots,n_d}^{\star,o}(z_1,\dots,z_{d})= &
(
\Li_{n_1,\dots,n_d}^{\star,o}(z_1,\dots, z_d),
\tilde\pi^{n_{1}}\Li_{n_2,\dots,n_d}^{\star,o}(z_2,\dots, z_d),  \\
&\dots,
\tilde\pi^{n_{1}+\cdots+n_{d-2}}\Li_{n_{d-1},n_d}^{\star,o}(z_{d-1},z_d),
\tilde\pi^{n_{1}+\cdots+n_{d-1}}\Li_{n_d}^{\star,o}(z_d)
)^\T
\in \C_\infty^d.
\end{align*}
\end{defn}

It turns out that 
the monodromy module $M_{n_1,\dots,n_d}^{\star, z_1,\dots,z_{d-1}}$
is the $A$-submodule of $\C_\infty^d$
generated by $d$ elements
\begin{align*}
\vec\Li_{n_1,\dots,n_d}^{\star,o}&(z_1,\dots, z_k) 
:=\tilde\pi^{n_{k+1}+\cdots+n_d}\cdot
\vec\LLii_{n_1,\dots,n_k}^{\star,o}(z_1,\dots,z_k)|_{t=\theta} \\
&=\tilde\pi^{n_{k+1}+\cdots+n_d}\cdot
\Bigl(
\Li_{n_1,\dots,n_k}^{\star,o} (z_1,\dots,z_k), 
\tilde\pi^{n_{1}}\Li_{n_2,\dots,n_k}^{\star,o}(z_2,\dots, z_k), \\
&\qquad\qquad
\dots,
\tilde\pi^{n_1+\dots+n_{k-1}}\Li_{n_k}^{\star,o}(z_k), 
\tilde\pi^{n_1+\dots+n_{k}},
\{0\}^{d-k-1}
\Bigr)^\T\in \C_\infty^d
\end{align*}
with $0\leqslant k\leqslant d-1$.
Actually it is independent of any choice of branches.


\begin{thm}\label{prop:ana con CMSPL}
{\rm (1).} $\vec\Li_{n_1,\dots,n_d}^\star(z_1,\dots,z_{d})$
is congruent to the tuple
$$
(
\Li_{n_1,\dots,n_d}^{\star}(z_1,\dots, z_d),
\tilde\pi^{n_{1}}\Li_{n_2,\dots,n_d}^{\star}(z_2,\dots, z_d), 
\dots,
\tilde\pi^{n_{1}+\cdots+n_{d-1}}\Li_{n_d}^{\star}(z_d)
)^\T
\in \C_\infty^d
$$
of \eqref{eq:CMPL}
modulo $M_{n_1,\dots,n_d}^{\star,z_1,\dots,z_{d-1}}$
when $(z_1,\dots,z_d)$ lies in $\mathbb D^\star$.

{\rm (2).}
$\vec\Li_{n_1,\dots,n_d}^{\star}(z_1,\dots,z_{d-1},-)$ 
locally admits an analytic  lift
(as a function on $z_d$)
in the sense of Proposition \ref{prop:ana con CPL}.
\end{thm}

\begin{proof}
The proof can be done in the same way to that of 
Theorem \ref{prop:ana con CMPL}.
\end{proof}

\section{Applications}\label{sec:applications}
By exploiting the techniques of
the continuation of multiple polylogarithms developed in the previous section, 
we explain  how
the logarithms associated with 
the tensor power of Carlitz module 
are extended 
to the whole space in \S \ref{sec:LogCn}.
We present the orthogonal property  of $t$-motivic CMPL and CMSPL
which extends the functional relations of Chang-Mishiba
in \S \ref{subsec:orthogonality}.
We show that Eulerian property
is independent of any choice of branches
in \S \ref{sec:CPY revisited}.

\subsection{Logarithms of tensor powers of Carlitz module}\label{sec:LogCn}
We  explain a method of continuation of
the logarithms associated with  tensor powers of the Carlitz module.

We begin with the review of the definition of $t$-modules (cf. \cite{BP}).
Let 
$\C_\infty\{\tau\}$ be the twisted polynomial algebra
in the variable $\tau$ over $\C_\infty$ with the relation
$$
\tau \alpha=\alpha^q\tau \qquad \text{for} \quad\alpha\in \C_\infty.
$$
An {\it $n$-dimensional $t$-module} $E$  over $\C_\infty$ is 
an $\F_q$-algebra homomorphism
$\rho_E:\F_q[t]\to \mathrm{Mat}_n(\C_\infty\{\tau\})$
such that for each $a\in \F_q[t]$,
$$
\rho_E(a)=\sum_i E_{a,i}\tau^i
$$
with $E_{a,i}\in \mathrm{Mat}_n(\C_\infty)$ and
$d\rho_E(a)-a\cdot I_n$
(where $d\rho_E(a)$ mean $E_{a,0}$) is a nilpotent matrix.
We denote the $t$-module whose action is given by $d\rho_E$  by $\Lie_E$.
One can show that there exists a unique $\F_q$-linear $n$-variable power series of the form
$\Exp_E=\tau^0+\sum_{i=1}^\infty\alpha_i\tau^i$ with $\alpha_i\in\mathrm{Mat}_n(\C_\infty) $
such that
$$
\Exp_E\circ d\rho_E(a)=\rho_E(a)\circ \Exp_E.
$$
The logarithm $\Log_E$ is defined to be the formal power series which is inverse to $\Exp_E$ and has the property
\begin{equation}\label{eq:log and action}
\Log_E\circ \rho_E(a)=d\rho_E(a)\circ \Log_E.
\end{equation}
We note that $\Exp_E$ converges everywhere on $\C_\infty^n$ while
$\Log_E$ 
converges on a certain milti-disk centered at the origin
(cf. \cite[Proposition/Definition 2.4.3]{AT90}).

For a positive integer $n$ we denote by $\mathrm C^{\otimes n}$
to be the $n$-th tensor power of the Carlitz module $\mathrm C$ (cf. \cite{AT90}).
It is given by an $\F_q$-algebra homomorphism
$\rho_n:\F_q[t]\to \mathrm{Mat}_n(\C_\infty\{\tau\}) $
determined by 
$\rho_n(t)=\theta I_n+N +E\tau$
with
$$
N=
\left(
    \begin{array}{cccc}
      0 & 1      &  &  \\
        &  \ddots & \ddots & \\
        &          & 0 & 1 \\
        &          &        & 0
    \end{array}
  \right), 
\qquad
E=
\left(
    \begin{array}{cccc}
      0 & 0 & \ldots & 0 \\
 \vdots & \vdots & \ddots & \vdots \\
      0 & 0 & \ldots & 0 \\
      1 & 0 & \ldots & 0
    \end{array}
  \right).
$$
The corresponding  $\Log_{\mathrm C^{\otimes n}}$
is an $\F_q$-linear map which satisfies
\begin{equation}\label{eq:log and action for Cn}
\Log_{\mathrm C^{\otimes n}}\circ (\theta I + N + E\tau)((z_1,\dots, z_n)^\T)
=(\theta I + N)\circ \Log_{\mathrm C^{\otimes n}}((z_1,\dots, z_n)^\T)
\end{equation}
in the region where  both the sides converge.
  Here $\T$ stands for the transpose.
In \cite[Proposition/Definition 2.4.3]{AT90},
it is shown that  the formal power series
$\Log_{\mathrm C^{\otimes n}}((z_1,\dots, z_n)^\T)$ converges when
\begin{equation}\label{eq:logn convergence condition}
|z_i|_\infty<|\theta|_\infty^{i-n+\frac{nq}{q-1}}
\qquad  (1\leqslant i \leqslant n).
\end{equation}
The continuation of $\Log_{\mathrm C^{\otimes n}}$ can be done as follows:
%
For any map $F:\TT\to\TT$ we define $L(F):\TT\to\TT$
by 
$$L(F)(Z)=tF(Z)-F(\theta Z)$$
for $Z\in\TT$.
Since we have
$t\vec\LLii_0(Z)\equiv \vec\LLii_0(tZ)$,  \ 
$\vec\LLii_0(Z^{(1)})\equiv \vec\LLii_0(Z)^{(1)}$ 
and
$\vec\LLii_0(Z+Z')\equiv \vec\LLii_0(Z)+\vec\LLii_0(Z')\ \bmod{\F_q[t]}$
for $Z, Z'\in\TT$  by our  construction in \S \ref{sec:prolong PL},
we have
$$
L(\vec\LLii_n)(z)=t\vec\LLii_n(z)-\vec\LLii_n(\theta z) 
\equiv \vec\LLii_n(tz)-\vec\LLii_n(\theta z)
\equiv \vec\LLii_n((t-\theta) z)\ \bmod{\Omega^{-n}\F_q[t]}
$$
for $z\in\C_\infty$.
Since
$$
L^i(\vec\LLii_n)(z)
\equiv \vec\LLii_n((t-\theta)^i z)
\ \bmod{\Omega^{-n}\F_q[t]},
$$
we have
\begin{align}
L^n(\vec\LLii_n)(z)
&\equiv \vec\LLii_n((t-\theta)^n z) 
\equiv  \Omega^{-n}\vec\LLii_0(\Omega^{n}(t-\theta)^n z) 
\equiv  \Omega^{-n}\vec\LLii_0((\Omega^{(-1)})^{n}z).  \notag \\
\intertext{By \eqref{eq:func eq for Li0}, we have}
&\equiv  \Omega^{-n}\{\vec\LLii_0(\Omega^{n} z^{(1)})+(\Omega^{(-1)})^nz\} \notag \\
&\equiv \vec\LLii_n(z^{(1)})+(t-\theta)^nz \ \bmod{\Omega^{-n}\F_q[t]} .
\label{eq: recursive LnLin}
\end{align}

By following \cite{CGM, GN}, we consider the map 
for $r\geqslant q$
$$
\delta_0^n:\TT_r\to \C_\infty^n\ (=\mathrm{Mat}_{n\times 1}(\C_\infty))
$$
sending each $f=\sum_{i\geqslant 0}c_i(t-\theta)^i\in\TT_r$ to
$(c_{n-1},\dots, c_1,c_0)^\T$.
By \cite[Proposition 2.5.5]{AT90}, we have
\begin{equation}\label{eq:delta0=lambda}
\delta_0^n(\Omega^{-n}\F_q[t])=\Lambda_n
\end{equation}
where $\Lambda_n$ is the $A$-module  under the $d\rho_n$-action
given by $\ker\Exp_{{\mathrm C}^{\otimes n}}$.
It induces the $\C_\infty$-linear map
$$
\delta_0^n:\TT_r/\Omega^{-n}\F_q[t]
\to \C_\infty^n/\Lambda_n.
$$
We have
$\vec\LLii_n(z)\in\TT_q/\Omega^{-n}\F_q[t]$
for each $z\in\C_\infty$ and
whence $L^i(\vec\LLii_n(z))\in\TT_q/\Omega^{-n}\F_q[t]$
for $i=1,\dots,n-1$.

\begin{defn}
Let $\vec e_k$ be the unit vector of $\C_\infty^n$ whose $k$-th coordinate is $1$.
We define the $\C_\infty$-linear map
$$
\vec\Log_n:\C_\infty^n\to \C_\infty^n/\Lambda_n
$$
by sending
$(z_1,\dots,z_n)^\T=\sum_{k=1}^nz_k\vec e_k$ to
$$
\sum_{k=1}^n\delta_0^n\left(L^{n-k}(\vec\LLii_n(z_k))\right)
\equiv\delta_0^n\circ \vec\LLii_n\left(\sum_{k=1}^n(t-\theta)^{n-k}z_k\right)
\quad \bmod \Lambda_n.
$$
\end{defn}

The following is an extension of the property  \eqref{eq:log and action for Cn}.

\begin{prop}\label{prop:loq modulo compatibility}
For $(z_1,\dots,z_n)\in\C_\infty^r$, we have
\begin{equation}\label{eq: log compatibility congruence}
\vec\Log_n ((\theta I_n +N+E\tau)(z_1,\dots,z_n)^\T)
\equiv (\theta I_n+N) \vec\Log_n ((z_1,\dots,z_n)^\T)
 \ \bmod \Lambda_n.
\end{equation}
\end{prop}

\begin{proof}
The right hand side is well-defined because we have $(\theta I_n+N)\Lambda _n\subset \Lambda_n$
by $\Lambda_n=\ker\Exp_{{\mathrm C}^{\otimes n}}$.
Put
$$
\vec\ell_k(z)=(\ell_{k,1}(z),\dots,\ell_{k,n}(z))^\T:=
\vec\Log_n (z\vec e_k)=
\delta_0^n\left(L^{n-k}(\vec\LLii_n(z))\right)
$$
in  $\C_\infty^n/\Lambda_n$
for $1\leqslant k \leqslant n$.
Since
$L(F)(z)=\{\theta+(t-\theta)\}F(z)-F(\theta z)$ for any map $F:\TT\to\TT$,
we have
\begin{align*}
&\vec\ell_i(z)=(\ell_{i1}(z),\dots,\ell_{in}(z))^\T
=\delta_0^n\left(L^{n-i}(\vec\LLii_n(z))\right)
=\delta_0^nL^{n-i-1}\left(L(\vec\LLii_n(z))\right)\\
&=\delta_0^nL^{n-i-1}\left(\{\theta+(t-\theta)\}\vec\LLii_n(z)
-\vec\LLii_n(\theta z)\right)
 \\
&\equiv (\theta \ell_{i+1,1}(z)+\ell_{i+1,2}(z)-\ell_{i+1,1}(\theta z),\dots,
\theta \ell_{i+1,n-1}(z)+\ell_{i+1,n}(z)-\ell_{i+1,n-1}(\theta z), \\
&\qquad\qquad
\theta \ell_{i+1,n}(z)-\ell_{i+1,n}(\theta z))^\T \\
&\equiv (\theta I_n +N)( \ell_{i+1,1}(z),\dots,\ell_{i+1,n}(z))^\T-(\ell_{i+1,1}(\theta z),\dots,\ell_{i+1,n}(\theta z))^\T  \\
&\equiv (\theta I_n +N)\vec\ell_{i+1}(z)-\vec\ell_{i+1}(\theta z)
\end{align*}
for $1\leqslant i<n$.
Actually the equation holds for $i=0$.
By \eqref{eq: recursive LnLin}, we also obtain
\begin{align*}
(\ell_{n1}(z^{(1)}),\dots,\ell_{nn}(z^{(1)}))^\T  
\equiv (\theta I_n +N)(\ell_{11}(z),\dots,\ell_{1n}(z))^\T
-(\ell_{11}(\theta z),\dots,\ell_{1n}(\theta z))^\T ,
\end{align*}
that is
$$
\vec\ell_{n}(z^{(1)})\equiv (\theta I_n +N)\vec\ell_{1}(z)-\vec\ell_{1}(\theta z).
$$
Therefore
\begin{align*}
&\vec\Log_n ( (\theta I_n +N+E\tau)(z_1,\dots,z_n)^\T) 
=\vec\Log_n((\theta z_1+ z_2,\dots, \theta z_{n-1}+z_n,\theta z_n+z_1^{(1)})^\T) \\
&=\sum_{i=1}^{n-1}\vec\ell_i(\theta z_i+z_{i+1})+\vec\ell_n(\theta z_n+z_1^{(1)}) 
=\sum_{i=1}^{n}\vec\ell_i(\theta z_i)+
\sum_{i=1}^{n-1}\vec\ell_i( z_{i+1})
+\vec\ell_n(z_1^{(1)}) \\
&\equiv \sum_{i=1}^{n-1}\vec\ell_{i+1}(\theta z_{i+1})+
\sum_{i=1}^{n-1}\vec\ell_i( z_{i+1})
+(\theta I_n+N)\vec\ell_1(z_1) 
=\sum_{i=1}^n(\theta I_n+N)\vec\ell_i(z_i)  \\
&=(\theta I_n+N) \vec\Log_n ((z_1,\dots,z_n)^\T).
\end{align*}
Thus we obtain the claim.
\end{proof}

The following theorem assures that 
$\vec\Log_n$ is an analytic continuation of $\Log_{\mathrm C^{\otimes n}}$.

\begin{thm}\label{thm:ana con log C}
\rm{(1).} $\Log_{\mathrm C^{\otimes n}}((z_1,\dots,z_n)^\T) \equiv
\vec\Log_n((z_1,\dots,z_n)^\T)
\bmod\Lambda_n$
when $(z_1,\dots,z_n)$ is
in the convergence region of \eqref{eq:logn convergence condition}.

\rm{(2).}
Let $\Exp_{n}:\C_\infty^n/\Lambda_n\to \C_\infty^n$ be the induced map from 
$\Exp_{\mathrm C^{\otimes n}}:\C_\infty^n\to \C_\infty^n$.
Then $\vec\Log_{n}$ is the inverse of $\Exp_{n}$.
 
\end{thm}

\begin{proof}
(1). 
We put 
$$\vec l_i(z)= (l_{i1}(z),\dots,l_{in}(z))^\T:=\Log_{\mathrm C^{\otimes n}}(z\vec e_i)$$
for each $i$.
By \eqref{eq:log and action for Cn}, we have
$$
\vec l_i(z)=(\theta I_n+N)\vec l_{i+1}(z)-\vec l_{i+1}(\theta z)
$$
for $1\leqslant i <n$.
Hence to show $\vec\ell_i(z)=\vec l_i(z)$ for all $i$, it is enough to prove
$\vec l_n(z)=\vec \ell_n(z)$. 
By \cite[Theorem  3.3.5]{CGM} and \cite[Theorem 4.14]{GN}, we have
$$
\Log_{\mathrm C^{\otimes n}}((0,\dots, 0, z)^\T)=\delta_0^n(\LLii_n(z))
$$
for  
$|z|_\infty<|\theta|_\infty^{\frac{nq}{q-1}}$,
which means 
$\vec l_n(z)=\vec \ell_n(z)$.
Hence our claim is proved.

\smallskip
(2). Let $\mathfrak z\in\C_\infty^n$.
Since the sequence ${\mathfrak z}_k=(\theta I + N)^{-k}\mathfrak z$ ($k=0,1,2,\dots$)
goes to $0\in\C_\infty^n$ and $\Exp_{\mathrm C^{\otimes n}}$ is continuous,
there is a ${\mathfrak z}_m$ such that
$\Exp_{\mathrm C^{\otimes n}}({\mathfrak z}_m)$
lies in  the region defined by 
\eqref{eq:logn convergence condition}.
Then we have
\begin{align*}
\vec\Log_{n}\circ\Exp_{\mathrm C^{\otimes n}} (\mathfrak z)
&=\vec\Log_{n}\circ\Exp_{\mathrm C^{\otimes n}}((\theta I + N)^{m}\mathfrak z_m)  \\
&=\vec\Log_{n}\circ(\theta I + N+E\tau)^{m}\circ\Exp_{\mathrm C^{\otimes n}}(\mathfrak z_m)  \\
&=(\theta I + N)^{m}\circ\vec\Log_{n}\circ\Exp_{\mathrm C^{\otimes n}}(\mathfrak z_m)  \\
&=(\theta I + N)^{m}\circ\Log_{\mathrm C^{\otimes n}}\circ\Exp_{\mathrm C^{\otimes n}}(\mathfrak z_m)  \\
&=(\theta I + N)^{m}(\mathfrak z_m) =\mathfrak z.
\end{align*}
Since $\Exp_{\mathrm C^{\otimes n}}: \C_\infty^n\to\C_\infty^n$
is a surjection with $\ker\Exp_{\mathrm C^{\otimes n}}=\Lambda_n$,
we get that $\vec\Log_{n}$ is the inverse of $\Exp_n$
\end{proof}

%
%

%

\begin{rem}
The logarithms of $t$-modules associated with Anderson-Thakur dual $t$-motive (\cite{AT90})
are discussed in \cite{CGM, CM, GN}.
They described  a certain special value of  their logarithms in terms of  CMSPL's.
The above logarithm $\Log_{\mathrm C^{\otimes n}}$ is regarded as the simplest case.
The author expects that their logarithms could be also analytically continued to the whole space
by elaborate description of the technical lemma in \cite[Lemma 4.2.1]{CGM}
in terms of CMSPL's and some sort of their relatives.
\end{rem}

\subsection{Orthogonality}\label{subsec:orthogonality}
The following functional relation was shown in \cite{GN}:
\begin{align*}
\LLii_{n_1,\dots,n_d}(Z_1,\dots,Z_d) &
=\sum_{i=2}^d (-1)^i\LLii_{n_{i-1},\dots,n_1}^\star(Z_{i-1},\dots,Z_1)
\LLii_{n_i,\dots,n_d}(Z_i,\dots,Z_d) \\
&+(-1)^{d+1}\LLii_{n_d,\dots,n_1}^\star(Z_d,\dots,Z_1)
\end{align*}
for $n_1,\dots,n_d\in\N$ and
$Z_1,\dots,Z_d\in\TT$ belonging to all the regions of convergence of
 each term.
The orthogonal property below is an extension of
the above relation to all branches:

\begin{thm}
Let $n_1,\dots,n_d\in\N$ and $Z_1,\dots,Z_d\in\TT$.
For any branch 
$\vec\LLii_{n_1,\dots,n_d}^{o}(Z_1,\dots,Z_d)\in \TT^d$ 
and
$\vec\LLii_{n_d,\dots,n_1}^{\ast,o}(Z_d,\dots,Z_1)\in \TT^d$,
we have
$$
\left(
\begin{array}{c}
\vec\LLii_{n_d,\dots,n_1}^{\star,o}(-Z_d,\dots,-Z_1) \\
\Omega^{-n_1-\cdots-n_d}     
\end{array}
\right)^\T
\cdot
\left(
\begin{array}{c}
\Omega^{-n_1-\cdots-n_d}      \\
\vec\LLii_{n_1,\dots,n_d}^{o}(Z_1,\dots,Z_d)  
\end{array}
\right)
\equiv 0 \
\bmod \F_q[[t]]\cdot\Omega^{-2(n_1+\cdots+n_d)}.
$$
Here $\vec\LLii_{n_d,\dots,n_1}^{\star,o}(-Z_d,\dots,-Z_1) $
means the vector putting
$\LLii_{n_i,\dots,n_1}^{\star,o}(-Z_i,\dots,-Z_1)=
(-1)^i\LLii_{n_i,\dots,n_1}^{\star,o}(Z_i,\dots,Z_1)$
for each $i$.
\end{thm}

\begin{proof}
Our proof is influenced by \cite[\S 4.2]{GN}.
Hereafter we fix $d$ generators 
$\vec\LLii_{n_1,\dots,n_d}^o(Z_1,\dots, Z_k)$
($0\leqslant k\leqslant d-1$) of
$\M_{n_1,\dots,n_d}^{Z_1,\dots,Z_{d-1}}$
as in Definition \ref{def:multiple LLii}.
By using their coordinates, we define the matrix  
$\Psi\in\GL_{d+1}(\TT)$ by
\begin{equation}\label{eq:Big Psi matrix}
{\footnotesize
\left(
    \begin{array}{ccccc}
       \Omega^{n_1+\dots+n_d} &    0   & 0 & \ldots & 0\\
      \Omega^{n_1+\dots+n_d}\LLii_{n_d}^o(Z_d) &  \Omega^{n_1+\dots+n_{d-1}} & 0 & \ldots &\vdots \\
       \vdots & \Omega^{n_1+\dots+n_{d-1}}\LLii_{n_{d-1}}^o(Z_{d-1})   &  & & \vdots \\
        \vdots  &  \vdots & \ddots &0 &  \vdots\\
       \vdots    & \vdots   & \ddots  & \Omega^{n_1} & 0 \\
      \Omega^{n_1+\dots+n_d}\LLii_{n_1,\dots,n_d}^o(Z_1,\dots,Z_d)  & 
      \Omega^{n_1+\dots+n_{d-1}}\LLii_{n_1,\dots,n_{d-1}}^o(Z_1,\dots,Z_{d-1})  &    \dots &  \Omega^{n_1}\LLii_{n_1}^o(Z_1)  & 1
    \end{array}
  \right).
  }
\end{equation}
Similarly we also fix 
$d$ generators 
$\vec\LLii_{n_d,\dots,n_1}^{\star,o}(Z_d,\dots, Z_{d-k})$
($1\leqslant k\leqslant d$) of
$\M_{n_d,\dots,n_1}^{\star, Z_d,\dots,Z_{2}}$
and define
$\Psi_\star\in\GL_{d+1}(\TT)$ by
\begin{equation}\label{eq:Big Psi star matrix}
\left(
    \begin{array}{ccccc}
       \Omega^{-n_1-\dots-n_d} &    0   & 0 & \ldots & 0\\
      \Omega^{-n_1-\dots-n_{d-1}}\LLii_{n_d}^{\star,o}(-Z_d) &  \Omega^{-n_1-\dots-n_{d-1}} & 0 & \ldots &\vdots \\
       \vdots & \Omega^{-n_1-\dots-n_{d-2}}\LLii_{n_{d-1}}^{\star,o}(-Z_{d-1})   &  & & \vdots \\
        \vdots  &  \vdots & \ddots &0 &  \vdots\\
       \vdots    & \vdots   & \ddots  & \Omega^{-n_1} & 0 \\
      \LLii_{n_d,\dots,n_1}^{\star,o}(-Z_d,\dots,-Z_1)  & 
      \LLii_{n_{d-1},\dots,n_{1}}^{\star,o}(-Z_{d-1},\dots,-Z_{1})  &    \dots & \LLii_{n_1}^{\star,o}(-Z_1)  & 1
    \end{array}
  \right).
\end{equation}
Here we note that  
$ \LLii_{n_d,\dots,n_1}^{\star,o}(-Z_d,\dots,-Z_1)=
(-1)^d \LLii_{n_d,\dots,n_1}^{\star,o}(Z_d,\dots,Z_1)$
by linearlity.
We consider the matrix $\Phi\in \Mat_{d+1}(\TT)$ given by 
\begin{equation}\label{eq:Big Phi matrix}
\Phi=
\left(
    \begin{array}{ccccc}
       (t-\theta)^{n_1+\dots+n_d} &    0   & 0 & \ldots & 0\\
       Z_d^{(-1)}(t-\theta)^{n_1+\dots+n_d} & (t-\theta)^{n_1+\dots+n_{d-1}} & 0 & \ldots& \vdots \\
       0 & Z_{d-1}^{(-1)}(t-\theta)^{n_1+\dots+n_{d-1}}   &  & & \vdots \\
       0   &  0 & \ddots &0 &  \vdots\\
       \vdots    & \vdots   & \ddots  & (t-\theta)^{n_1} & 0 \\
       0 &  \dots &    0 &  Z_1^{(-1)}(t-\theta)^{n_1}  & 1
    \end{array}
  \right).
\end{equation}
Then by Proposition \ref{lem:property of multiple LLii}
we have
$$\Psi=\Phi^{(1)}\Psi^{(1)}.$$ 
While by Proposition \ref{lem:property of multiple LLii star},
we also have 
$$\Psi_\star^{(1)}=\Psi_\star\Phi^{(1)}.$$ 
Therefore
$$
(\Psi_\star\Psi)^{(1)}=
\Psi_\star^{(1)}\Psi^{(1)}=
\Psi_\star\Phi^{(1)}\Psi^{(1)}=
\Psi_\star\Psi.
$$
Thus
$
\Psi_\star\Psi\in\GL_{d+1}(\F_q[[t]]).
$
By calculating its $(d+1,1)$-component, we obtain the claim. 
\end{proof}

Chang-Mishiba functional relation (\cite[Lemma 4.2.1]{CM}) is 
\begin{align*}
\Li_{n_1,\dots,n_d}(z_1,\dots,z_d) &
=\sum_{i=2}^d (-1)^i\Li_{n_{i-1},\dots,n_1}^\star(z_{i-1},\dots,z_1)
\Li_{n_i,\dots,n_d}(z_i,\dots,z_d) \\
&+(-1)^{d+1}\Li_{n_d,\dots,n_1}^\star(z_d,\dots,z_1)
\end{align*}
for  $n_1,\dots,n_d\in\N$ and
$z_1,\dots,z_d\in\C_\infty$ belonging to all the regions of convergence of
 all terms.
It is extended to all branches as  follows:

\begin{cor}
Let $n_1,\dots,n_d\in\N$ and $z_1,\dots,z_d\in\C_\infty$.
For any branch 
$\vec\Li_{n_1,\dots,n_d}^{o}(z_1,\dots,z_d)$ 
and
$\vec\Li_{n_d,\dots,n_1}^{\star,o}(z_d,\dots,z_1)\in \C_\infty^d$,
we have
$$
\left(
\begin{array}{c}
\vec\Li_{n_d,\dots,n_1}^{\star,o}(-z_d,\dots,-z_1) \\
\tilde\pi^{n_1+\cdots+n_d}     
\end{array}
\right)^\T
\cdot
\left(
\begin{array}{c}
\tilde\pi^{n_1+\cdots+n_d}      \\
\vec\Li_{n_1,\dots,n_d}^{o}(z_1,\dots,z_d)  
\end{array}
\right)
\equiv 0 \
\bmod \tilde\pi^{2(n_1+\cdots+n_d)}A.
$$
\end{cor}

\begin{proof}
By Lemma \ref{lem:AS}.(2), 
Proposition \ref{lem:property of multiple LLii}.(3),
Proposition \ref{lem:property of multiple LLii star}.(3) 
and $\Omega(\theta)\neq 0$,
$t=\theta$ is inside the regions of convergence of 
$\LLii_{n_i,\dots,n_j}^{\star,o}(z_i,\dots,z_j)$ and
$\LLii_{n_j,\dots,n_i}^{\star,o}(-z_j,\dots,-z_i)$
($1\leqslant i\leqslant j \leqslant d$).
So we have
$
\Psi_\star\Psi\in\GL_{d+1}(\F_q[t]),
$
which implies
$$
\left(
\begin{array}{c}
\vec\LLii_{n_d,\dots,n_1}^{\star,o}(-z_d,\dots,-z_1) \\
\Omega^{-n_1-\cdots-n_d}     
\end{array}
\right)^\T
\cdot
\left(
\begin{array}{c}
\Omega^{-n_1-\cdots-n_d}      \\
\vec\LLii_{n_1,\dots,n_d}^{o}(z_1,\dots,z_d)  
\end{array}
\right)
\in 
\Omega^{-2(n_1+\cdots+n_d)}\cdot\F_q[t].
$$
By evaluating $t=\theta$, we obtain the claim.
\end{proof}

\subsection{Eulerian property}\label{sec:CPY revisited}

We discuss Eulerian properties of the special values of 
multiple polylogarithm at algebraic points.
We show that Eulerian property for CMPL and CMSPL
is independent of any choice of branches.

\begin{defn}
Let $n_1,\dots, n_d\in\N$ and $z_1,\dots,z_d\in \C_\infty$, 
Put  $\Li_{n_1,\dots, n_d}^o(z_1,\dots,z_d)\in\C_\infty$
be an branch, that is, the last coordinate of an appropriate branch
$\vec\Li_{n_1,\dots, n_d}^o(z_1,\dots,z_d)\in\C_\infty^d$.
It is called {\it Eulerian} when
$\Li_{n_1,\dots, n_d}^o(z_1,\dots,z_d)/ \tilde\pi^{n_1+\cdots+n_d}\in K$.
We may say the same thing for 
$\Li_{n_1,\dots, n_d}^{\star,o}(z_1,\dots,z_d)$.
\end{defn}

\begin{thm}\label{thm:extended CPY}
Put $n_1,\dots, n_d\in\N$ and $z_1,\dots,z_d\in \bar K$. 
Let $\vec\LLii_{n_i,\dots, n_d}^o(z_1,\dots,z_d)$
be an branch
with coordinates satisfying
\begin{equation}\label{eq:assume}
\Li_{n_i,\dots, n_d}^o(z_i,\dots,z_d)
:=\LLii_{n_i,\dots, n_d}^o(z_i,\dots,z_d)(\theta)\neq 0
\end{equation}
for all $i=1,2,\dots,d$.
If $\Li_{n_1,\dots, n_d}^o(z_1,\dots,z_d)$ 
is Eulerian,
then so is any other branch 
$\Li_{n_1,\dots, n_d}^{o'}(z_1,\dots,z_d)$.
\end{thm}

\begin{proof}
Though the proof goes in the same way as that of \cite[Theorem 2.5.2]{CPY},
we repeat here for the proof of Theorem \ref{thm:extended CPY}.
We consider the matrix
$\Phi\in\mathrm{Mat}_{d+1}(\bar K[t])$ 
given in \eqref{eq:Big Phi matrix}
with $Z_i=z_i\in\bar K$
and the vector
\begin{align*}
\psi&=(1,\Omega^{n_1+\cdots+n_d},\Omega^{n_1+\cdots+n_d}\LLii_{n_d}^o(z_d),\dots, 
\Omega^{n_1+\cdots+n_d}\LLii_{n_2,\dots,n_d}^o(z_2,\dots,z_d), \\
&\qquad\qquad\qquad\qquad\qquad\qquad\qquad\qquad\qquad\qquad
\Omega^{n_1+\cdots+n_d}\LLii_{n_1,\dots,n_d}^o(z_1,\dots,z_d)
)^\T
\\
&=\Bigl(1,
\Omega^{n_1+\cdots+n_{d}},
\Omega^{n_1+\cdots+n_{d-1}} \LLii_{\{0\}^{1}}^o(\Omega^{n_d}z_d),
\dots, 
\Omega^{n_1}\LLii_{\{0\}^{d-1}}^o(\Omega^{n_2}z_2,\dots, \Omega^{n_d}z_d), \\
&\qquad\qquad\qquad\qquad\qquad\qquad\qquad\qquad\qquad\qquad\qquad
\LLii_{\{0\}^{d}}^o(\Omega^{n_1}z_1,\dots, \Omega^{n_d}z_d)
\Bigr)^\T
\end{align*}
in $\mathrm{Mat}_{d+2,1}(\TT)$.
We have the difference equation
\begin{equation}\label{difference equation}
\psi^{(-1)}=
\left(
    \begin{array}{cc}
         1 &   \\ 
          & \Phi
    \end{array}
\right)
\psi.
\end{equation}
Proposition \ref{lem:property of multiple LLii} assures that $\psi$ is in $\mathrm{Mat}_{d+2,1}(\E)$
for $z_1,\dots, z_d\in\bar K$.
By the ABP criteria (\cite[Theorem 3.1.1]{ABP}),
there exists
$(f_0,\dots,f_{d+1})\in\mathrm{Mat}_{1,d+2}(\bar K[t])$
such that 
\begin{equation*}
(f_0,\dots,f_{d+1})
\psi=0
\end{equation*}
and
whose specialization at $t=\theta$ yields 
Eulerian property of
$\Li_{n_1,\dots,n_d}^o(z_1,\dots,z_d)$.
Particularly we have $f_{0}(\theta)\neq 0$,   $f_{d+1}(\theta)\neq 0$
and
$f_1(\theta)=\dots=f_{d}(\theta)=0$.
Put 
$$
(B_0,B_1,\dots,B_d,0):=
\left(\frac{f_0}{f_{d+1}}, \dots, \frac{f_d}{f_{d+1}},1\right) 
-\left(\frac{f_0}{f_{d+1}}, \dots, \frac{f_d}{f_{d+1}},1\right)^{(-1)} 
\cdot \left(
    \begin{array}{cc}
         1 &   \\ 
          & \Phi
    \end{array}
\right)
$$
in $\mathrm{Mat}_{1,d+2}(\bar K(t)).$

The equation $
(f_0,\dots,f_{d+1})
\psi=0$ implies
$$
(B_0,B_1,\dots,B_d,0)\psi=0,
$$
that means
\begin{align*}
B_0+
B_1\Omega^{n_1+\cdots+n_{d}} +
& B_2\Omega^{n_1+\cdots+n_{d-1}} \LLii_{\{0\}^{1}}^o(\Omega^{n_d}z_d)+
\cdots  \\
&\qquad\qquad
+B_{d}\Omega^{n_1}\LLii_{\{0\}^{d-1}}^o(\Omega^{n_2}z_2,\dots, \Omega^{n_d}z_d)=0.
\end{align*}
While we have
$$
\LLii_{\{0\}^{d+1-i}}^o(\Omega^{n_i}z_i,\dots, \Omega^{n_d}z_d)(\theta^{q^n})\neq 0
$$
for all $i=1,2,\dots,d$ and $n\geqslant 1$
by our assumption \eqref{eq:assume}, Proposition \ref{lem:property of multiple LLii}.(4) and $\Omega(\theta)\neq 0$.
By combining it with $\Omega(\theta^{q^n})=0$ for all $n$
and
$B_i\in\bar K(t)$,
we recursively obtain
$$
B_0=B_1=\dots=B_d=0.
$$
Put 
$D=\left(
    \begin{array}{cccc}
         1 &    &  & \\ 
            & 1  & & \\
            && \ddots & \\
         \delta_1 & \dots & \delta_d & 1 
    \end{array}
    \right)
$ in $\mathrm{Mat}_{1,d+1}(\bar K(t))$
with $\delta_i=\frac{f_i}{f_{d+1}}\in \bar K(t)$ ($i=1,\dots,d$). 
Then we have
$$
D^{(-1)}\Phi =
\left(
    \begin{array}{cc}
         \Phi' &   \\ 
          & 1
    \end{array}
\right)D
$$
where $\Phi'$ is the upper  left $d\times d$-part of $\Phi$.

Hereafter we fix $d$ generators 
$
\vec\LLii_{n_1,\dots,n_d}^o(z_1,\dots, z_k)
$
($0\leqslant k\leqslant d-1$) of
$\M_{n_1,\dots,n_d}^{z_1,\dots,z_{d-1}}$
as in Definition \ref{def:multiple LLii}.
By using their coordinates, we define the matrix  
$\Psi\in\mathrm{Mat}_{d+1}(\E)\cap \GL_{d+1}(\TT)$ 
given in \eqref{eq:Big Psi matrix}
with $Z_i=z_i\in\bar K$.
It satisfies
$$
\Psi^{(-1)}=\Phi\Psi.
$$
Thus we have
$$
(D\Psi)^{(-1)}=
\left(
    \begin{array}{cc}
         \Phi' &   \\ 
          & 1
    \end{array}
\right)
D\Psi.
$$

While we have the difference equation
$
\left(
    \begin{array}{cc}
         \Psi' &   \\ 
          & 1
    \end{array}
\right)^{(-1)}
=
\left(
    \begin{array}{cc}
         \Phi' &   \\ 
          & 1
    \end{array}
\right)
\left(
    \begin{array}{cc}
         \Psi' &   \\ 
          & 1
    \end{array}
\right)
$
where 
$\Psi'$ means the upper  left $d\times d$-part  of $\Psi$
since $\Phi$ is a lower triangular matrix.
Then by \cite[\S 4.1.6]{P} there exist $\nu_1,\dots,\nu_d\in\F_q(t)$
such that 
$$
D\Psi=\left(
    \begin{array}{cc}
         \Psi' &   \\ 
          & 1
    \end{array}
\right)
\left(
    \begin{array}{cccc}
         1 &    &  & \\ 
            & 1  & & \\
            && \ddots & \\
         \nu_1 & \dots & \nu_d & 1 
    \end{array}
    \right).
$$
The equation implies  on the last row
\begin{align*}
\nu_1=&\delta_1\Omega^{n_1+\dots+n_d}+\delta_2\Omega^{n_1+\dots+n_{d-1}}
\LLii_{\{0\}^{1}}^o(\Omega^{n_d}z_d)+\cdots \\
&\qquad
+\delta_d\Omega^{n_1}\LLii_{\{0\}^{d-1}}^o(\Omega^{n_2}z_2,\dots, \Omega^{n_d}z_d) 
+\LLii_{\{0\}^{d}}^o(\Omega^{n_1}z_1,\dots, \Omega^{n_d}z_d), \\
\nu_2=&\delta_2\Omega^{n_1+\dots+n_{d-1}}+\delta_3\Omega^{n_1+\dots+n_{d-2}}
\LLii_{\{0\}^{1}}^o(\Omega^{n_{d-1}}z_{d-1})+\cdots \\
& \qquad
+\delta_d\Omega^{n_1}\LLii_{\{0\}^{d-2}}^o(\Omega^{n_2}z_2,\dots,\Omega^{n_{d-1}}z_{d-1}) 
+\LLii_{\{0\}^{d-1}}^o(\Omega^{n_1}z_1,\dots, \Omega^{n_{d-1}}z_{d-1}) ,
\\
& \qquad \vdots \\
\nu_d=& \delta_{d}\Omega^{n_1}+\LLii_{\{0\}^{1}}^o(\Omega^{n_1}z_1). 
\end{align*}
By $\nu_i\in\F_q(t)$, we have $\nu_i(\theta^{q^n})=\nu_i(\theta)^{q^n}$
for all $i=1,\dots,d$ and $n\geqslant 0$.
By Proposition \ref{lem:property of multiple LLii}.(4) and
$\Omega(\theta^{q^n})=0$ for $n\geqslant 1$,
we obtain
$$
\nu_i(\theta)^{q^n}=
\LLii_{\{0\}^{i}}^o(\Omega^{n_1}z_1,\dots, \Omega^{n_i}z_i)(\theta)^{q^n}
$$
for infinitely many $n$.
By taking $q^n$-th root of both hand sides, we see that
$\nu_i(\theta)=\LLii_{\{0\}^{i}}^o(\Omega^{n_1}z_1,\dots, \Omega^{n_i}z_i)(\theta)
=\tilde\pi^{-n_1-\cdots-n_i}\Li_{n_1,\dots,n_i}^o(z_1,\dots,z_i)$
is in $K$.

By Remark \ref{rem:other branch of CMPL},
any other branch 
$\Li_{n_1,\dots, n_d}^{o'}(z_1,\dots,z_d)$
is given by the form
$$
\Li_{n_1,\dots, n_d}^{o}(z_1,\dots,z_d)+\sum_{i=0}^{d-1}
\alpha_i\cdot\tilde\pi^{n_{i+1}+\cdots+n_d}\Li_{n_1,\dots,n_i}^o(z_1,\dots,z_i)
$$
with $\alpha_i\in A$.
Whence we get that $\Li_{n_1,\dots, n_d}^{o'}(z_1,\dots,z_d)$ is also Eulerian.
\end{proof}


\begin{rem}
In 
\cite[Theorem 4.3.2]{CPY}, 
$z_1,\dots,z_d$ are assumed to be in $\bar K^\times\cap \mathbb D$.
And 
\cite[Theorem 2.5.2]{CPY}
is shown for  $z_1,\dots,z_d\in\bar K[t]$
satisfying \eqref{eq:condition for Z} and \eqref{eq:assume}
for $(n_i,\dots,n_j)$ with $1\leqslant i\leqslant j\leqslant d$.
\end{rem}

The branch independency also holds for the star version.

\begin{thm}\label{thm:extended CPY star}
Put $n_1,\dots, n_d\in\N$ and $z_1,\dots,z_d\in \bar K$. 
Let $\vec\LLii_{n_i,\dots, n_d}^{\star,o}(z_1,\dots,z_d)$
be a branch
with coordinates satisfying
\begin{equation}\label{eq:assume star}
\Li_{n_i,\dots, n_d}^{\star,o}(z_i,\dots,z_d)
:=\LLii_{n_i,\dots, n_d}^{\star,o}(z_i,\dots,z_d)(\theta)\neq 0.
\end{equation}
If $\Li_{n_1,\dots, n_d}^{\star,o}(z_1,\dots,z_d)$ 
is Eulerian,
then so is any other branch
$\Li_{n_1,\dots, n_d}^{\star,o'}(z_1,\dots,z_d)$.
\end{thm}

\begin{proof}
The proof goes in the same way to that of Theorem \ref{thm:extended CPY}.
We work over the matrices of the star dual $t$-motives constructed in  \cite{GN}.

Fix $d$ generators 
$\vec\LLii_{n_1,\dots,n_d}^{\star,o}(z_1,\dots, z_k)$
($0\leqslant k\leqslant d-1$) of
$\M_{n_1,\dots,n_d}^{\star, z_1,\dots,z_{d-1}}$
as in Definition \ref{def:multiple LLii}.
By using their coordinates, we define the matrix  
$\Psi^\ast\in\mathrm{Mat}_{d+1}(\E)\cap \GL_{d+1}(\TT)$ 
given in \eqref{eq:Big Psi matrix}
with $\LLii_{n_i,\dots,n_j}^o(Z_i,\dots,Z_j)$
replaced with 
$\LLii_{n_i,\dots,n_j}^{\star,o}(-z_i,\dots,-z_j)$,
and  the matrix	
$\Phi^\star\in \mathrm{Mat}_{d+1}(\bar K[t])$ given by
\begin{equation*}
{
\left(
    \begin{array}{cccc}
       (t-\theta)^{n_1+\dots+n_d} &    0   & \ldots & 0\\
       -z_d^{(-1)}(t-\theta)^{n_1+\dots+n_d} & (t-\theta)^{n_1+\dots+n_{d-1}} &     
        & \vdots \\
       z_{d-1}^{(-1)}z_d^{(-1)}(t-\theta)^{n_1+\dots+n_d} & -z_{d-1}^{(-1)}(t-\theta)^{n_1+\dots+n_{d-1}}   & \ddots &\vdots \\
       \vdots    & \vdots   & & 0  \\
      (-1)^dz_1^{(-1)}\cdots z_d^{(-1)}(t-\theta)^{n_1+\dots+n_d} & 
      (-1)^{d-1}z_1^{(-1)}\cdots z_{d-1}^{(-1)}(t-\theta)^{n_1+\dots+n_{d-1}}  &   \ldots   & 1
    \end{array}
  \right).
}
\end{equation*}
Then we have 
\begin{equation*}\label{eq:diff eq for star}
\Psi^{\star (-1)}=\Phi^\star \Psi^\star
\end{equation*}
(cf. \cite[Remark 5.1]{GN}).
Put the vector
\begin{align*}
\psi^\star&=(1,\Omega^{n_1+\cdots+n_d},\Omega^{n_1+\cdots+n_d}\LLii_{n_d}^{\star,o}(-z_d),\dots, 
\Omega^{n_1+\cdots+n_d}\LLii_{n_2,\dots,n_d}^{\star,o}(-z_2,\dots,-z_d), \\
&\qquad\qquad\qquad\qquad\qquad\qquad\qquad\qquad\qquad\qquad
\Omega^{n_1+\cdots+n_d}\LLii_{n_1,\dots,n_d}^{\star,o}(-z_1,\dots,-z_d)
)^\T
\end{align*}
in $\mathrm{Mat}_{d+2,1}(\TT)$.
The ABP criteria (\cite[Theorem 3.1.1]{ABP}) assures 
the existence of
$(f^\star_0,\dots,f_{d+1}^\star)\in\mathrm{Mat}_{1,d+2}(\bar K[t])$
such that 
\begin{equation*}
(f^\star_0,\dots,f_{d+1}^\star)
\psi^\star=0.
\end{equation*}
It follows
$$
\left(\delta_0^\star, \dots, \delta_d^\star,1\right) 
-\left(\delta_0^\star, \dots, \delta_d^\star,1\right)^{(-1)} 
\cdot \left(
    \begin{array}{cc}
         1 &   \\ 
          & \Phi^\star
    \end{array}
\right)
=(0,\dots,0)
$$
in $\mathrm{Mat}_{1,d+2}(\bar K(t))$
with $\delta_i^\star=f_i^\star/f_{d+1}^\star$.
Then by \cite[\S 4.1.6]{P} there exists  $\nu_1^\star,\dots,\nu_d^\star\in\F_q(t)$
such that 
$$
\left(
    \begin{array}{cccc}
         1 &    &  & \\ 
            & 1  & & \\
            && \ddots & \\
         \delta_1^\star & \dots & \delta_d^\star & 1 
    \end{array}
    \right)
\Psi^\star=\left(
    \begin{array}{cc}
         \Psi'^\star &   \\ 
          & 1
    \end{array}
\right)
\left(
    \begin{array}{cccc}
         1 &    &  & \\ 
            & 1  & & \\
            && \ddots & \\
         \nu_1^\star & \dots & \nu_d^\star & 1 
    \end{array}
    \right).
$$
The equation on the last row implies
that $\nu_i^\star(\theta)=\LLii_{\{0\}^{i}}^{\star,o}(-\Omega^{n_1}z_1,\dots, -\Omega^{n_i}z_i)(\theta)
=\tilde\pi^{-n_1-\cdots-n_i}\Li_{n_1,\dots,n_i}^{\star,o}(-z_1,\dots,-z_i)
=(-1)^i\tilde\pi^{-n_1-\cdots-n_i}\Li_{n_1,\dots,n_i}^{\star,o}(z_1,\dots,z_i)$
is in $K$.

Since 
any other branch 
$\Li_{n_1,\dots, n_d}^{\star,o'}(z_1,\dots,z_d)$
is given by the form
$$
\Li_{n_1,\dots, n_d}^{\star,o}(z_1,\dots,z_d)+\sum_{i=0}^{d-1}
\alpha_i\cdot\tilde\pi^{n_{i+1}+\cdots+n_d}\Li_{n_1,\dots,n_i}^{\star,o}(z_1,\dots,z_i)
$$
with $\alpha_i\in A$,
we get that $\Li_{n_1,\dots, n_d}^{\star,o'}(z_1,\dots,z_d)$ is also Eulerian.
\end{proof}


\begin{rem}
We note that Theorems \ref{thm:extended CPY} and \ref{thm:extended CPY star}
hold without the assumption \eqref{eq:assume} and \eqref{eq:assume star}
respectively when $d=1$.
One could hope that both theorems hold  without imposing the assumptions
though they are  required to apply  \cite[Theorem 2.5.2]{CPY}
in the proofs of the theorems.
\end{rem}

%


\begin{rem}
In $p$-adic situation (in characteristic $0$ case), 
Coleman established the theory of $p$-adic iterated integrals
in \cite{Col}
by making essential use of crystalline Frobenius automorphisms
and
his theory  enables us to carry out an analytic continuation of $p$-adic multiple polylogarithms (cf.\cite{Fpmzv1}).
While 
in positive characteristic case,
a theory of iterated integrals does not seem to have been established yet,
however 
the author expects that there would exist such a theory
where
Artin-Schreier equations 
alternatively play a crucial role. 
\end{rem}

\bigskip
{\it Acknowledgments.}
The author has been supported by JSPS KAKENHI JP18H01110.
He is grateful to 
Chieh-Yu Chang,
Oguz Gezmis,
Yoshinori Mishiba,
Federico Pellarin,
and the referee
who gave valuable comments 
on the paper.



\begin{thebibliography}{99}


\bibitem[ABP]{ABP}
Anderson, G.W., Brownawell, W.D. and Papanikolas, M.A., 
{\it Determination of the algebraic relations among special $\Gamma$-values in positive characteristic},
Ann. of Math. (2) {\bf 160} (2004), no. 1, 237--313.

\bibitem[AT90]{AT90}
Anderson, G.W. and Thakur, D.S.,
{\it Tensor powers of the Carlitz module and
zeta values}, Annals of Math. {\bf 132} (1990), 159--191.

\bibitem[AT09]{AT09}
Anderson, G.W. and Thakur, D.S.,
{\it Multizeta values for $\F_q[t]$, their period interpretation, and relations between them}, Int. Math. Res. Not. IMRN 2009, no. 11, 2038--2055.

\bibitem[BP]{BP}
Brownawell, W.D. and Papanikolas, M.A.,
{\it A rapid introduction to Drinfeld modules, $t$-modules, and $t$-motives},
$t$-Motives: Hodge structures, transcendence, and other motivic aspects,
European Mathematical Society, Z\"{u}rich, 2020, 3--30. 

\bibitem[C]{C14}
Chang, C.-Y.,
{\it Linear independence of monomials of multizeta values in positive characteristic}, Compositio Math. {\bf 150} (2014), 1789--1808.

\bibitem[CGM]{CGM}
Chang, C.-Y.,  Green, N. and Mishiba, Y.,
{\it Taylor coefficients of Anderson-Thakur series and explicit formulae},
to appear in Math Ann.

\bibitem[CM]{CM}
Chang, C.-Y. and Mishiba, Y.,
{\it  On multiple polylogarithms in characteristic $p$: $v$-adic vanishing versus $\infty$-adic Eulerianness,}
Int. Math. Res. Notices. IMRN (2019), no. 3, 923--947.


\bibitem[CPY]{CPY}
Chang, C.-Y., Papanikolas. M., and Yu, J, 
{\it An effective criterion for Eulerian multizeta values in positive characteristic}, J. Eur. Math. Soc. (JEMS){\bf 21} (2019) , no. 2, 405--440.

\bibitem[Co]{Col}
Coleman, R.,
{\it Dilogarithms, regulators and $p$-adic $L$-functions},
Invent. Math. {\bf 69} (1982), no. 2, 171--208. 





\bibitem[F]{Fpmzv1} 
Furusho, H.,
{\it $p$-adic multiple zeta values I -- $p$-adic multiple polylogarithms and the $p$-adic KZ equation},
Inv. Math, {\bf 155}, no. 2, 253-286, (2004).


%

%

\bibitem[GN]{GN}
Green, N and  Ngo Dac, T.,
{\it On log-algebraic identities for Anderson t-modules and characteristic p multiple zeta values},
{\tt arXiv:2007.11060}, preprint (2020).


%

\bibitem[P]{P}
Papanikolas, M. A.,
{\it Tannakian duality for Anderson-Drinfeld motives and algebraic independence of Carlitz logarithms},
Invent. Math. {\bf 171} (2008), no. 1, 123--174.

\bibitem[T]{T-book}
Thakur, D.S.,
{\it Function field arithmetic},
World Scientific Publishing Co., Inc., River Edge, NJ, 2004.

\bibitem[Z]{Z}
Zhao, J.,
{\it Analytic continuation of multiple polylogarithms},
Analysis Mathematica. {\bf 33}, 301--323 (2007).

\end{thebibliography}
\end{document}